\documentclass{amsart}

\usepackage{amsmath,amsthm,amssymb,enumerate}
\usepackage{hyperref}
\usepackage{color}
\newtheorem{assumption}{Assumption}[section]

\newtheorem{theorem}{Theorem}[section]
\newtheorem{lemma}[theorem]{Lemma}

\theoremstyle{definition}
\newtheorem{definition}[theorem]{Definition}
\newtheorem{example}[theorem]{Example}

\newtheorem{corollary}[theorem]{Corollary}

\theoremstyle{remark}
\newtheorem{remark}[theorem]{Remark}

\numberwithin{equation}{section}

\newcommand\bD{\mathbb{D}}
\newcommand\bG{\mathbb{G}}
\newcommand\bR{\mathbb{R}}

\newcommand\bB{\mathbf{B}}
\newcommand\bA{\mathbf{A}}

\newcommand\fra{\mathfrak{a}}
\newcommand\frb{\mathfrak{b}}

\newcommand\frD{\mathfrak{D}}

\newcommand\frL{\mathfrak{L}}
\newcommand\frM{\mathfrak{M}}
\newcommand\frF{\mathfrak{F}}
\newcommand\frG{\mathfrak{G}}

\newcommand\opar{\text{\raise.2ex\hbox{${\scriptstyle | }$}\kern-.34em$($} }
\newcommand\cbrk{\text{$]$\kern-.15em$]$}} 

\newcommand\cB{\mathcal{B}}

\newcommand\cF{\mathcal{F}}

\newcommand\cH{\mathcal{H}}

\newcommand\cK{\mathcal{K}}

\newcommand\cN{\mathcal{N}}
\newcommand\cP{\mathcal{P}}

\newcommand{\E}{\mathbb{E}}

\newcommand{\R}{\mathbb{R}}
\newcommand{\G}{\mathbb{G}}

\begin{document}

\title[Localization errors]{Localization errors in solving  stochastic partial differential equations in the whole space}

\author[M. Gerencs\'er]{M\'at\'e Gerencs\'er}
\address{School of Mathematics and Maxwell Institute,
The University of Edinburgh, 
The King's Buildings, 
Peter Guthrie Tait Road, 
EDINBURGH, 
EH9 3FD
Scotland, UK}
\email{m.gerencser@sms.ed.ac.uk}

\author[I. Gy\"ongy]{Istv\'an Gy\"ongy}
\address{School of Mathematics and Maxwell Institute,
The University of Edinburgh, 
The King's Buildings, 
Peter Guthrie Tait Road, 
EDINBURGH, 
EH9 3FD
Scotland, UK}
\email{gyongy@maths.ed.ac.uk}


\date{}

\dedicatory{}

\begin{abstract}
Cauchy problems with SPDEs on the whole space are localized to 
Cauchy problems on a ball of radius $R$.  This localization reduces various kinds of spatial approximation schemes to finite dimensional problems. The error is shown to be exponentially small. 
As an application, a numerical scheme is presented which combines 
the localization and the space and time discretization, and thus is fully implementable. 
\end{abstract}

\maketitle

\section{Introduction}

Parabolic, possibly degenerate, linear stochastic partial differential equations (SPDEs) are considered. In applications such equations are often given on the whole space, which makes the direct implementation of discretization methods problematic. 
Finite element methods, see e.g., \cite{DP}, \cite{Hau}, \cite{JK}, \cite{W}, \cite{Y} and their references, mostly treat problems on bounded domains and often under strong restrictions on the differential operators, denoted by $L$ and $M$ below. For finite difference schemes convergence results are available on the whole space, see e.g. \cite{Yoo} for the non-degenerate and \cite{DK}, \cite{Gy} for the degenerate case,  but the schemes themselves are infinite systems of equations. A natural way to overcome this difficulty is to ``cut off'' the equation outside of a large ball of radius $R$. A similar approach to obtain error estimates for a truncated terminal condition in deterministic PDEs of optimal stopping problems is used in \cite{Siska}. 

The main results of the present paper are 
Theorems \ref{theorem cutoff} and \ref{theorem main2}. 
Theorem \ref{theorem cutoff} compares solutions of two SPDEs whose data agrees on a ball of radius $R$ and establishes an error estimate of order $e^{-\delta R^2}$ in the supremum norm. 
The proof relies on transforming fully degenerate SPDEs to zero order equations via the method of characteristics, whose analysis goes back to \cite{Kun} and \cite{KR}, 
see also the recent work \cite{LM} and the references therein. 
Once such a result is used to estimate the difference between 
the original equation and its truncation, one can approximate 
the truncated equation with known numerical schemes. 
Our choice is the finite difference method for the spatial 
and the implicit Euler method for the temporal approximation. 
The analysis of the former is invoked from \cite{Gy}, 
while for the latter one requires an error estimate for the time 
discretization of the finite difference scheme, 
which is at this point a finite dimensional SDE. 
Of course this estimate needs to be independent of the spatial mesh size, 
and this is established in Theorem \ref{theorem time error}.  
To the authors' best knowledge such an analysis of a 
full discretization is also a new result for degenerate SPDEs, 
and in fact even in the deterministic case it improves the results 
of \cite{DK} in that we are not restricted to 
finite difference schemes which are monotone.
An error estimate for the approximations obtained by 
localized and fully discretized SPDEs is given 
in Theorem \ref{theorem main1}, which is a special case of Theorem \ref{theorem main2}, where 
the accuracy of the approximations is improved by Richardson extrapolation in the spatial 
discretization.

We mention that an alternative method for localization is to introduce artificial Dirichlet boundary conditions on a large ball, this approach is used for deterministic PDEs in, for example, 
in \cite{LL}, \cite{HRSW}, and to some extent, in \cite{Siska}. The order of error in these works is $e^{-\delta R}$, and so the method of the present paper yields an improvement even in the deterministic case. In the present context the method of artificial boundary condition would present additional issues. For instance, if the original SPDE is degenerate, then after introducing the boundary conditions the resulting equation may not even have a solution. Even if we do assume non-degeneracy, in the generality considered here - which is justified and motivated by the filtering equation in signal-observation models - there is no approximation theory of SPDEs on bounded domains with Dirichlet boundary condition, and therefore such a localization method does not help in finding an efficient numerical scheme. One indication of the problem is that regardless of the smoothness of the data and of the boundary, solutions of Dirichlet problems for SPDEs in general do not have continuous second derivatives, see \cite{KDom}.

Let us introduce some notation used throughout the paper. 
All random elements will be given on a fixed
  probability space 
$(\Omega,\cF,P)$, equipped with a filtration 
$(\cF_t)_{t\geq0}$   of  $\sigma$-fields
$\cF_{t}\subset\cF$. We suppose that
this probability space carries  
a sequence of independent Wiener processes $(w^k)_{k=1}^{\infty}$,  
adapted to the filtration $(\cF_t)_{t\geq0}$, such that 
$w^k_t-w^k_s$ is independent of $\cF_s$ 
for each $k$ and any $0\leq s\leq t$.
It is assumed that $\cF_0$ contains all $P$-null subsets of $\Omega$, 
 so that
$(\Omega,\cF,P)$ is a complete probability space and the $\sigma$-fields
$\cF_{t}$ are complete. 
By $\mathcal P$ we denote the predictable
$\sigma$-field of subsets of $\Omega\times[0,\infty)$
generated by  $(\cF_t)_{t\geq0}$. We use the shorthand notation $\E^\alpha X=[\E(X)]^\alpha.$

For $p\in[2,\infty)$  and $\vartheta\in\bR$ we denote by $L_{p,\vartheta}(\bR^{d},\cH)$ 
the space of measurable mappings 
$f$ from $\bR^d$ into a separable Hilbert space $\cH$, such that 
$$
|f|_{L_{p,\vartheta}}=
\big(\int_{\bR^d}|(1+|x|^2)^{\vartheta/2}f(x)|_{\cH}^p\,dx\big)^{1/p}<\infty. 
$$
We do not include
the symbol $\cH$ in the notation of the norm
in $L_{p,\vartheta}(\bR^{d},\cH)$. Which $\cH$ is involved will be clear 
from  the context. We  do the same in
other similar situations.
In this paper $\cH$ will be $l_2$ or $\bR$.  
The space of functions from $L_{p,\vartheta}(\bR^{d},\cH)$, 
whose generalized derivatives 
up to order $m$ are also in $L_{p,\vartheta}(\bR^{d},\cH)$, is denoted by 
$W^m_{p,\vartheta}(\bR^{d},\cH)$. 
By definition $W^0_{p,\vartheta}(\bR^{d},\cH)=L_{p,\vartheta}(\bR^{d},\cH)$. 
The norm $|f|_{W^m_{p,\vartheta}}$ of $f$ in $W^m_{p,\vartheta}(\bR^{d},\cH)$ 
is defined 
by  
\begin{equation}
                                          \label{4.11.1}
|f|^p_{W^m_{p,\vartheta}}=\sum_{|\alpha|\leq m}|D^{\alpha}f|^p_{L_{p,\vartheta}}, 
\end{equation}
where 
$D^{\alpha}:=D_1^{\alpha_1}...D_d^{\alpha_d}$ for 
{\it multi-indices} 
$\alpha:=(\alpha_1,...,\alpha_d)\in\{0,1,...\}^d$ of 
{\it length} $|\alpha|:=\alpha_1+\alpha_2+...+\alpha_d$,  
and $D_if$ is the generalized derivative of $f$ with respect to 
$x^i$ 
for $i=1,2...,d$. We also use the notation $D_{ij}=D_iD_j$ and   
$Df=(D_1f,...,D_df)$. 
When we talk about  ``derivatives up to order $m$" of a function  
for some nonnegative integer $m$, then we always include the zeroth-order derivative, 
i.e. the function itself. 
 Unless otherwise indicated at some places, 
the summation convention with respect to repeated integer 
valued indices is used throughout the paper. The constants in our estimates, usually denoted by $N$, may change from line to line in the calculations, but their dependencies will always be clear from the statements.

\section{Formulation and localization error estimate}

Consider the stochastic equation
\begin{equation}                                                \label{SPDE general}                                                             
du_t(x)=(L_tu_t(x)+f_t(x))\,dt+
\sum_{k=1}^{\infty}(M^k_tu_t(x)+g^k_t(x))\,dw^k_t  
\end{equation}
on $(t,x)\in[0,T]\times\bR^d=:H_T$, with initial condition
\begin{equation}					      \label{SPDE general-ini}
u_0(x)=\psi(x), \quad x\in\bR^d.     
\end{equation}
Here $f$ and $g=(g^k)_{k=1}^{\infty}$ are functions 
on 
$\Omega\times H_T$ with values in $\bR$ and $l_2$, 
respectively, 
and  
$L$ and $M^k$ are second order and first order differential 
operators of the form 
$$
L_t=a^{ij}_t(x)
D_iD_j+b^i_t(x)D_i+c_t(x),
\quad
M^k_t=\sigma^{ik}_t(x)D_i+\mu^k_t(x), \quad k=1,2,..., 
$$
where the coefficients $a^{ij}$, $b^i$, $c$, $\sigma^{ik}$ 
and $\mu^k$ are real-valued functions on $\Omega\times H_T$
for $i,j=1,2,...d$, and integers $k\geq1$. 

For an integer $m\geq0$, $p\in[2,\infty)$,  and $\vartheta\in\bR$  
the following assumptions ensure the existence and uniqueness 
of a $W^m_{p,\vartheta}$-valued solution $(u_t(\cdot))_{t\in[0,T]}$.

\begin{assumption}                                                         \label{assumption parabolicity}
For $P\otimes dt\otimes dx$-almost all $(\omega,t,x)\in\Omega\times[0,T]\times\bR^d$ 
$$
 \alpha^{ij}_t(x)z^iz^j\geq 0
$$
for all $z\in\bR^d$, where 
$$
\alpha^{ij}=2a^{ij}-\sigma^{ik}\sigma^{ jk}. 
$$
\end{assumption}

This is a standard assumption in the theory of stochastic 
PDEs. Below we assume some smoothness on $\alpha$ in $x$, and 
on all the coefficients and free terms of problem 
\eqref{SPDE general}-\eqref{SPDE general-ini}. We will also 
require that the nonnegative symmetric square root $\rho:=\sqrt{\alpha}$ 
possesses bounded second order derivatives in $x$. 
Concerning this assumption we remark that it is 
well-known from \cite{F} that $\rho$ is Lipschitz continuous  
in $x$ if $\alpha$ is bounded and has bounded second 
order derivatives, but it is also known that the second order 
derivatives of $\rho$ may not exist in the classical sense, 
even if $\alpha$ is smooth with bounded derivatives 
of arbitrary order. 

\begin{assumption}                     \label{assumption regularitycoeff}
(a) The derivatives in $x\in\bR^d$ of $a^{ij}$ up to order $\max(m,2)$ 
are $\mathcal P\otimes\mathcal B(\bR^d)$-measurable 
functions, bounded by $K$ for all $i,j\in\{1,2,...,d\}$. 

(b) The derivatives in $x\in\bR^d$ of $b^{i}$ and $c$ 
up to order $m$ are $\mathcal P\otimes\mathcal B(\bR^d)$-measurable 
functions, bounded by $K$ for all $i\in\{1,2,...,d\}$. 
The functions   $\sigma^{i}
=(\sigma^{ik})_{k=1}^{\infty}$ and $\mu=(\mu^{k})_{k=1}^{\infty}$ 
are $l_2$-valued and their derivatives in $x$ 
up to order $m+1$ are 
$\mathcal P\otimes\mathcal B(\bR^d)$-measurable 
$l_2$-valued functions, bounded by $K$.

(c) The derivatives in $x\in\bR^d$ of $\rho=\sqrt{\alpha}$ up to order $m+1$ 
are $\mathcal P\otimes\mathcal B(\bR^d)$-measurable 
functions, bounded by $K$.
\end{assumption}

\begin{assumption}                       \label{assumption regularitydata}
The initial value, 
$\psi$ is an $\cF_0$-measurable random variable 
with values in $W^m_{p,\vartheta}$. 
The free data, $ f_t $ and $ g_t =(g^k)_{k=1}^{\infty}$ are 
predictable 
processes with values in $W^m_{p,\vartheta}$ and 
$W^{m+1}_{p,\vartheta}(l_2) $, respectively, 
such that almost surely
\begin{equation}			\label{jan22}	
\mathcal K_{m,p,\vartheta}^p(T):=|\psi|_{W^m_{p,\vartheta}}^p+\int_0^T
\big(|f_t|^p_{W^m_{p,\vartheta}}+|g_t|^p_{W^{m+1}_{p,\vartheta} }\big)\,dt<\infty.
\end{equation}
\end{assumption}

\begin{definition}                                                               \label{definition solution}
A $W_{p,\vartheta}^1$-valued function $u$,
defined on  $[0,T]\times\Omega$, 
is called a solution 
of 
\eqref{SPDE general}-\eqref{SPDE general-ini}
on $[0,T]$ if $u$ is predictable on $[0,T]\times\Omega$, 
$$
\int_0^{T}|u_t|_{W^1_{p,\vartheta}}^p\,dt<\infty\,(a.s.), 
$$
and for each $\varphi\in C_0^{\infty}(\bR^d)$ 
for almost all $\omega\in\Omega$
\begin{align}
(u_t,\varphi)=&(\psi,\varphi)
+\int_0^t\{-(a^{ij}_sD_iu_s,D_j\varphi)
+(\bar b^{i}_sD_iu_s+c_su_s+f_s,\varphi)\}\,ds     \nonumber\\
&
+\int_0^t(\sigma^{ir}_sD_iu_s
+\mu^{r}_{s}u_s+g^r_s,\varphi)\,dw^r_{s}                \label{solution}
\end{align}
for all $t\in[0,T]$,   where $\bar
b^i=b^{i}-D_ja^{ij}$, and $(\cdot\,,\cdot)$ denotes the inner
product in the Hilbert space of square integrable real-valued
functions on $\bR^d$.  
\end{definition}
 
The following theorem follows from Theorem 2.1 in \cite{GGK}. 
\begin{theorem}                                                                      \label{theorem ex}
Let Assumptions \ref{assumption parabolicity}, 
\ref{assumption regularitycoeff} (a)-(b), 
and \ref{assumption regularitydata} 
with $m\geq0$ hold. Then there exists at most one solution 
 on $[0,T]$. 
If Assumptions \ref{assumption parabolicity}, 
\ref{assumption regularitycoeff}(a), 
and \ref{assumption regularitydata} 
hold with $m\geq1$, then 
there exists a unique solution $u=(u_t)_{t\in[0,T]}$ on $[0,T]$.  
Moreover, $u$ is a  $W^{m}_{p,\vartheta}$-valued 
weakly continuous process,  
it is a strongly continuous 
process with values in $W^{m-1}_{p,\vartheta}$, and for every $q>0$ 
and $n\in\{0,1,...,m\}$
\begin{equation}                                                                      \label{estimate}
\E\sup_{t\in[0,T]}|u_t|_{W^n_{p,\vartheta}}^q
\leq N\E\mathcal K^q_{n,p,\vartheta}(T), 
\end{equation}
where $N$ is a constant depending only on $K$, $T$, $d$, $m$, $p$, $\vartheta$, and $q$.  
\end{theorem}
\begin{remark}                                                                   \label{remark poly}
Theorem 2.1 in \cite{GGK} covers the $\vartheta=0$ case. 
We can reduce the case of 
general $\vartheta$ to this by  rewriting the equation for $u$ as an equation 
for $\tilde{u}(x)=u(x)(1+|x|^2)^{-\vartheta/2}$ (see e.g. \cite{GG}). It is easily seen that the coefficients of the resulting equation still satisfy the conditions of the theorem.
\end{remark}

\begin{definition}
A $\mathcal P\otimes\mathcal B(\bR^d)$-measurable random field 
$u$ on $H_T$ is called a classical solution 
of \eqref{SPDE general}-\eqref{SPDE general-ini}, 
if  along with its derivatives in $x$ up to order 2 it is 
continuous in $(t,x)\in H_T$, 
it satisfies 
\eqref{SPDE general}-\eqref{SPDE general-ini}  
almost surely for all $(t,x)\in H_T$,  
and there exists a finite random variable $\xi$ and a constant $s$ 
such that almost surely
$$
|D^{\alpha}u_t(x)|\leq \xi(1+|x|)^{s} 
\quad
\text{for all $(t,x)\in H_T$ and for $|\alpha|\leq2$.}
$$ 
\end{definition}

\begin{corollary}                                          \label{corollary ex}
Let assumptions of Theorem \ref{theorem ex} hold 
with 
$m>2+d/p$. 
Then there exists a unique classical solution $u$ 
to \eqref{SPDE general}-\eqref{SPDE general-ini}. 
\end{corollary}

Let us refer to the problem 
\eqref{SPDE general}-\eqref{SPDE general-ini} 
as Eq($\frD$), where $\frD$ stands for the ``data"
$$
\frD=(\psi, a, b, c,\sigma,\mu, f,g)
$$
with $a=(a^{ij})$, $b=(b^i)$, $\sigma=(\sigma^{ki})$, $g=(g^k)$ 
and $\mu=(\mu^{k})$. 
We are interested in the error when instead of 
Eq($\frD$) we solve Eq($\bar\frD$) 
with 
$$
\bar\frD
=(\bar\psi, \bar a, \bar b, 
\bar c,\bar\sigma,\bar\mu, \bar f,\bar g).
$$

\begin{assumption}                             \label{assumption R}
Almost surely 
\begin{equation}
\frD=\bar\frD\quad \text{on $[0,T]\times\{x\in\R^d:|x|\leq R\}$}. 
\end{equation}
\end{assumption}

The main example to keep in mind is 
when each component of $\bar\frD$ is 
a truncation of the corresponding component of $\frD$ 
(see Section \ref{Section-findiff} below). 
Let 
$$
B_R=\{x\in\R^d:|x|\leq R\}
$$
 for $R>0$. 
Define $\bar\cK_{m,p,\vartheta}^p(T)$ as 
$\mathcal K_{m,p,\vartheta}^p(T)$ with $\bar \psi$, $\bar f$ and $\bar g$ 
in place of $\psi$, $f$ and $g$, respectively.
The main result reads as follows.
\begin{theorem}                                   \label{theorem cutoff}
Let $\nu\in(0,1)$ and let Assumptions 
\ref{assumption parabolicity},
\ref{assumption regularitycoeff} (b)-(c)
and 
\ref{assumption regularitydata} hold 
with $m>2+d/p$ and $\vartheta\in\R$ for $\frD$ and $\bar \frD$. Let also Assumption
\ref{assumption R}
hold. Then 
Eq($\frD$) and Eq($\bar\frD$) have a unique classical 
solution $u$ and $\bar u$, respectively, and for $q>0$, $r>1$
\begin{equation}                                             \label{main estimate}
\E\sup_{t\in[0,T]}\sup_{x\in B_{\nu R}}|u_t(x)-\bar u_t(x)|^q
\leq 
Ne^{-\delta R^2}
\E^{1/r}
(\cK_{m,p,\vartheta}^{qr}(T)+\bar\cK_{m,p,\vartheta}^{qr}(T)), 
\end{equation}
where $N$ and $\delta$ are positive constants, depending on 
$K$, $d$, $T$,  $q$,  $r$, $\vartheta$,  $p$, and $\nu$. 
\end{theorem}

First we collect some auxiliary results. 
The following lemma is a version of Kolmogorov's 
continuity criterion, see Theorem 3.4 of \cite{GL}. 

\begin{lemma}											\label{thm GL}
Let $x(\theta)$ be a stochastic process parametrized by 
and continuous in $\theta\in D\subset\R^p$, where $D$ is a 
direct product of lower dimensional closed balls. 
Then for all $0<\alpha<1$, $q\geq1$, and $s>p/\alpha$, 
$$\E\sup_{\theta}|x(\theta)|^q\leq N(1+|D|)\left[\sup_{\theta}(\E|x(\theta)|^{qs})^{1/s}
+\sup_{\theta\neq\theta'}\left(
\frac{\E|x(\theta)-x(\theta')|^{qs}}{|\theta-\theta'|^{qs\alpha}}\right)^{1/s}\right]$$
where $N=N(q,s,\alpha,p)$, 
and $|D|$ is the volume of $D$.
\end{lemma}

\begin{lemma}                                     \label{lemma flow exponential}
Let $(\alpha_t)_{t\in[0,T]}$ 
and $(\beta_t)_{t\in[0,T]}$ 
be $\cF_t$-adapted processes with values in $\bR^d$ 
and $l_2(\bR^d)$, respectively, 
in magnitude bounded by a constant $K$. 
Then for the process
\begin{equation}                                        \label{070414}	
X_t=\int_0^t\alpha_s\,ds
+\int_0^t\beta^k_s\,dw^k_s,\quad t\in[0,T]
\end{equation}
there exist constants  
$\varepsilon=\varepsilon(K,T)>0$ 
and a $N=N(K,T)$ such that 
$$
\E\sup_{t\leq T} e^{\varepsilon|{X}_{t}|^2}\leq N. 
$$
\end{lemma}

\begin{proof} 
A somewhat more general lemma is proved in \cite{Siska}. 
For convenience of the reader we give the proof here.
By It\^o's formula 
$$
Y_{t}:=e^{|X_t|^2e^{-\mu t}}
=1+\int_0^{t} e^{|X_s|^2
e^{-\mu s}-\mu s}\{|\beta_s|^2+2\alpha_s X_s
$$
$$
+2|\beta_s X_s|^2-\mu |X_s|^2\}\,ds+m_{t}
$$
for any $\mu\in\bR$, where $(m_t)_{t\in[0,T]}$ 
is a local martingale starting from 0.  
By simple inequalities 
$$
2\alpha X+2|\beta X|^2
\leq |\alpha|^2+|X|^2+2|\beta|^2 |X|^2\leq K^2+
(2K^2+1)|X|^2. 
$$
Hence for 
$\mu=(2K^2+1)$ and for a stopping time $\tau\leq T$
we have 
$$
\E Y_{t\wedge\tau_n}\leq 1+2K^2\int_0^t \E Y_{s\wedge\tau_n}\,ds,
$$
for $\tau_n=\tau\wedge\rho_n$, where $(\rho_n)_{n=1}^{\infty}$
is a localizing sequence of stopping times for $m$. 
Hence, by Gronwall's lemma,
$$
\E Y_{t\wedge\tau_n}\leq e^{2K^2T}\,.
$$
where $N$ is independent of $n$. 
Letting here $n\rightarrow\infty$, by Fatou's lemma 
we get 
$$
\E e^{|X_{\tau}|^2e^{-\mu T}}
\leq \E e^{|X_{\tau}|^2e^{-\mu \tau}}\leq e^{K^2T}
$$
for stopping times $\tau\leq T$. Hence applying 
Lemma 3.2 from \cite{GK2003} for $r\in(0,1)$ we 
obtain 
$$
\E\sup_{t\leq T}e^{r|X_{\tau}|^2e^{-\mu T}}
\leq \tfrac{2-r}{1-r}e^{rK^2T}\,.
$$
\end{proof}
To formulate our next lemma we consider the stochastic 
differential equation
\begin{equation}                                  \label{SDE}
dX_s=\alpha_s(X_s)\,ds+\beta^k_s(X_s)\,dw_s^k, 
\end{equation}
where $\alpha$ and $\beta=(\beta^k)$ are 
$\cP\otimes\cB(\bR^d)$-measurable function 
on $\Omega\times[0,T]\times\bR^d$, 
with values in $\bR^d$ and $l_2(\bR^d)$ 
such that they are bounded in magnitude by $K$ and satisfy 
the Lipschitz condition in $x\in\bR^d$ with a Lipschitz 
constant $M$, uniformly in the other 
arguments. Then equation \eqref{SDE} with initial condition 
$X_t=x$ has a unique solution $X^{t,x}=(X^{t,x}_s)_{s\in[t,T]}$ 
for any $t\in[0,T]$ and $x\in\bR^d$.

\begin{remark}
It is well known from \cite{Kun} that the solution of \eqref{070414} 
can be chosen to be continuous in $t,x,s$. In the following, 
by $X^{t,x}_s$ we always understand such a continuous modification.
\end{remark}

\begin{lemma}													\label{lemma flow sup}
Set $\hat{X}^{t,x}={X}^{t,x}-x$.
There exists a constant 
$\delta=\delta(d,K,M,T)>0$ such that for any $R$,
\begin{equation}\label{lemma eq1}
\E\sup_{0\leq t\leq s\leq T}\sup_{|x|\leq R}e^{|\hat{X}^{t,x}_s|^2\delta}\leq N(1+R^{d+1/2}),
\end{equation}
and for any $R$ and $r$
\begin{equation}\label{lemma eq2}
P(\sup_{0\leq t\leq s\leq T}\sup_{|x|\leq R}|\hat{X}_s^{t,x}|>r)\leq Ne^{-\delta r^2}(1+R^{d+1/2}),
\end{equation}
where $N=N(d,K, M,T)$.
\end{lemma}
\begin{proof}
It is easy to see that \eqref{lemma eq1} implies \eqref{lemma eq2}, so we need only prove the former. For a fixed $\delta$, to be chosen later, let us use the notations $f(y)=e^{|y|^2\delta}$ and $\gamma=2(d+2)+1$. By Lemma \ref{thm GL}, we have
$$
\E\sup_{0\leq t\leq s\leq T}\sup_{|x|\leq R}f(\hat{X}_s^{t,x})\leq N(1+R^d)\sup_{0\leq t\leq s\leq T}\sup_{|x|\leq R}(\E f^{\gamma}(\hat{X}_s^{t,x}))^{1/\gamma}
$$
\begin{equation}													\label{040414}
+N(1+R^d)
\mathop
{\sup_{0\leq t\leq s\leq T}}_{0\leq t'\leq s'\leq T}
\mathop{\sup_{|x|\leq R}}_{|x'|\leq R}
\left(
\frac{\E|f(\hat{X}^{t,x}_s)-f(\hat{X}^{t',x'}_{s'})|^{\gamma }}
{(|t-t'|^2+|s-s'|^2+|x-x'|^2)^{\gamma /4}}\right)^{1/\gamma}.
\end{equation}
The first term above, by Lemma \ref{lemma flow exponential}, 
provided $\delta\leq \varepsilon/\gamma$, can be estimated by $NR^d$. 
As for the second one,
$$
f(\hat{X}^{t,x}_s)-f(\hat{X}^{t',x'}_{s'})
=\int_0^1\partial f(\vartheta\hat{X}^{t,x}_s
+(1-\vartheta)\hat{X}^{t',x'}_s)(\hat{X}^{t,x}_s-\hat{X}^{t',x'}_{s'})
\,d\vartheta.$$
Notice that $|\nabla f(y)|\leq N(\delta)f^2(y)$, 
therefore by Jensen's inequality and Lemma \ref{lemma flow exponential} 
 again, provided $\delta\leq \varepsilon/(8\gamma)$, we obtain
$$
\E|f(\hat{X}^{t,x}_s)-f(\hat{X}^{t',x'}_{s'})|^{\gamma}
\leq N\E^{1/2}|\hat{X}^{t,x}_s-\hat{X}^{t',x'}_{s'}|^{2\gamma}.
$$
Now the the right-hand side can be estimated by 
standard moment bounds for SDEs, see e.g. 
Corollary 2.5.5 
in \cite{K2}, from which we obtain
$$\left(
\frac{\E|f(\hat{X}^{t,x}_s)-f(\hat{X}^{t',x'}_{s'})|^{2\gamma }}
{(|t-t'|^2+|s-s'|^2+|x-x'|^2)^{\gamma /2}}\right)^{1/(2\gamma)}\leq N(1+R^{1/2}).$$

\end{proof}

\emph{Proof of Theorem \ref{theorem cutoff}.} 
 
Throughout the proof we will use the constant $\lambda=\lambda(d,q)$, 
which stands for a power of $R$, and, like $N$ and $\delta$, 
may change from line to line. Clearly it suffices to prove 
Theorem \ref{theorem cutoff}
with $e^{-\delta R^2}R^{\lambda}$ in place of 
$e^{-\delta R^2}$ in the right-hand side 
of inequality \eqref{main estimate}. 
We also assume first that $q>1$ and $\vartheta=0$. 

The main idea of the proof is based on stochastic representation of 
solutions to linear stochastic PDEs of parabolic type, see \cite{KR},  \cite{Kun}, and \cite{LM}.  This representation can be viewed as the generalization of the well-known Feynman-Kac formula and is derived as follows. First, we consider an equation which differs from the original only by an additional stochastic term driven by an independent Wiener process. The new equation is fully degenerate and taking conditional expectation with respect to the original filtration of its solution gives back $u$. On the other hand, the method of characteristics allows us to transform the fully degenerate equation to a much simpler one. This provides a formula for the representation of $u$, and, more importantly for our purposes, allows us to compare $u$ and $\bar u$ on the level of characteristics.

Recall that $\rho=(\rho^{ir}_t(x))_{i,r=1}^d$ is  
the symmetric nonnegative 
 square root of 
$\alpha=(2a^{ij}-\sigma^{ik}\sigma^{jk})_{i,j=1}^d$ and $\bar\rho$ is the symmetric nonnegative 
 square root of $\bar\alpha=(2\bar a^{ij}-\bar\sigma^{ik}\bar\sigma^{jk})_{i,j=1}^d$. 
 Then due to Assumption 
 \ref{assumption R},  $\rho=\bar\rho$ almost surely
 for all $t\in[0,T]$ and for $|x|\leq R$.
 Let $(\hat w_t^r)_{t\geq0,r=1\ldots d}$ be a $d$-dimensional Wiener 
process, also independent of the $\sigma$-algebra 
$\cF_{\infty}$ generated by $\cF_t$ 
for $t\geq0$. Consider the problem
 \begin{align}
 dv_t(x)=&(L_tv_t(x)+f_t(x))\,dt
 +(M^k_tv_t(x)+g^k_t(x))\,dw^k_t        \nonumber\\
 &+\cN^r_tv_t(x)\,d\hat w^r_t                 \label{v}\\
 v_0(x)=&\psi(x),                                   \label{v ini}
 \end{align}
 where $\cN^r=\rho^{ri}D_i$. 
Then by Corollary \ref{corollary ex}, \eqref{v}-\eqref{v ini} has a unique classical solution $v$, 
and for each $t\in[0,T]$ and $x\in\bR^d$ almost surely 
\begin{equation}                                      \label{vetites}
u_t(x)=\E(v_t(x)|\cF_t). 
\end{equation}
Together with \eqref{v} let us consider the stochastic 
differential equation 
 \begin{equation}                                                                    \label{flow}
dY_t=\beta_t(Y_t)\,dt-\sigma_t^k(Y_t)\,dw^k_t
-{\rho}^r_t(Y_t)\,d\hat{w}^r_t,\quad 0\leq t\leq T, \quad Y_0=y,  
\end{equation} 
where  
$$
\beta_t(y)=-b_t(y)+\sigma^{ik}_{t}(y)D_i\sigma^k_t(y)
+\rho^{ri}_t(y)D_i\rho^r_t(y)+\sigma^k_t(y)\mu^k_t(y), 
\quad t\in[0,T],\,y\in\bR^d,
$$
and $\sigma^k$, $\rho^r$ stand for the column vectors $(\sigma^{1k},\ldots,\sigma^{dk})$, $(\rho^{1r},\ldots,\rho^{dr})$, respectively.
By the It\^o-Wentzell formula from \cite{K}, for 
$$
U_t(y):=v_t(Y_t(y))
$$  
 we have (to ease the notation 
we omit the parameter $y$ in 
$Y_t(y)$)
$$
dv_t(Y_t)=(L_tv_t(Y_t)+f_t(Y_t))\,dt
+(M^k_tv_t(Y_t)+g^k_t(Y_t))\,dw^k_t
+\cN^r_t v_t(Y_t)\,d\hat w^r_t
$$
$$
+(\beta^i_tD_iv_t(Y_t)+a^{ij}_tD_{ij}v_t(Y_t))\,dt
-\sigma^{ik}_tD_iv_t(Y_t)\,dw^k_t
-\cN^r_tv(Y_t)\,d\hat{w}^r_t
$$
\begin{equation}                                                                       \label{IW0}
-\sigma_t^{ik}D_i(M^k_tv_t(Y_t)+g^k(Y_t))\,dt
-\cN^r_t\cN^r_tv_t(Y_t)\,dt.   
\end{equation}
Due to cancellations on the right-hand side of 
\eqref{IW0} we obtain 
\begin{align*}                                                      \label{IW2}
dU_t(y)=&\{\gamma_t(Y_t(y))U_t(y)+\phi_t(Y_t(y))\}\,dt\\
&+\{\mu^k_t(Y_t(y))U_t(y)+g^k_t(Y_t(y))\}\,dw^k_t, 
\quad  U_0(y)=\psi(y), 
\end{align*}
where 
$$
\gamma_t(x):=c_t(x)-\sigma^{ki}_t(x)D_i\mu^k_t(x), 
\quad 
\phi_t(x)=f_t(x)-\sigma^{ki}_t(x)D_ig^k_t. 
$$
Notice 
that in the special case when 
$f=0$, $g=0$, $c=0$, $\mu=0$ and $\psi(x)=x^i$ for 
$i\in\{1,...,d\}$,  we get 
$\tilde v^i_t(Y_t(y))=y^i$ 
for $i=1,...,d$, where $\tilde v^i$ is the solution 
of  \eqref{v}-\eqref{v ini} with $f=c=0$, $g=\mu=0$, $\sigma=0$ and 
$\psi(x)=x^i$. 
Hence for each $t\in[0,T]$ 
the mapping $y\to Y_t(y)\in\bR^d$ has an inverse, 
$Y_t^{-1}$, for almost every $\omega$, 
and the mapping $x\to\tilde v_t(x)=(\tilde v^{i}_t(x))_{i=1}^d$, 
defined  
by the continuous random field 
$(\tilde v_t^{i})_{(t,x)\in H_T}$ gives 
a continuous modification of $Y^{-1}_t$. 
Also, we can write $v_t(x)=U_t(Y^{-1}_t)$. 

As we shall see, due to the data being the same on a large ball, the characteristics $Y$ and $\bar Y$ agree on an event of large probability. This fact and the above representation will yield the estimate \eqref{main estimate}. Set $\bar U_t(y)=\bar v_t(\bar Y_t(y))$, where $\bar v_t(x)$ 
and $\bar Y_t(y)$ are defined as $v_t(x)$ and $Y_t(y)$ 
in \eqref{v}-\eqref{v ini} and \eqref{flow}, respectively, 
with $\bar\frD$ and $\bar\rho$ in place of 
$\frD$ and $\rho$. 

 Introduce the notations $\bB_R=[0,T]\times B_R$ and 
$\bA_R=\bB_R\cap\mathbb{Q}^{d+1}$. Since $u$ and $\bar u$ 
are continuous in both variables,
\begin{equation}                                              \label{rational}
\sup_{(t,x)\in\bB_{\nu R}}|u_t(x)-\bar u_t(x)|
=\sup_{(t,x)\in \bA_{\nu R}}|u_t(x)-\bar u_t(x)|
\end{equation}
Let $\nu'=(1+\nu)/2$ and define the event 
$$
H:=\left[\sup_{(t,x)\in \bB_{\nu R}}|Y_t^{-1}(x)|>\nu'R\right]
\cup\left[\sup_{(t,x)\in \bB_{\nu' R}}|Y_t(x)|>R\right]. 
$$
Then 
$$
H^c=\left[Y_t^{-1}(x)\in B_{\nu' R}, \forall (t,x)\in \bB_{\nu R}\right]
\cap\left[Y_t(x)\in B_R, \forall (t,x)\in \bB_{\nu' R}\right],
$$
and thus on $H^c$
$$
Y_t(x)=\bar Y_t(x)\quad \text{for $(t,x)\in \bB_{\nu' R}$}, 
$$
$$
Y_t^{-1}(x)=\bar Y_t^{-1}(x)\quad \text{for $(t,x)\in \bB_{\nu R}$ },  
$$
and consequently, 
\begin{equation*}
v_t(x)=\bar v_t(x)\quad \text{for $(t,x)\in \bB_{\nu R}$}. 
\end{equation*}
Therefore, by \eqref{vetites} and \eqref{rational}, 
and by Doob's, H\"older's, and the conditional Jensen inequalities,
$$
\E \sup_{(t,x)\in\bB_{\nu R}}|u_t(x)-\bar u_t(x)|^q 
\leq 
\E \sup_{t\in[0,T]\cap\mathbb{Q}}
|\E({\bf1}_H\sup_{(\tau,x)\in\bA_{\nu R}}
|v_{\tau}(x)-\bar v_{\tau}(x)||\mathcal{F}_t)|^q
$$
\begin{equation}                                                                \label{intermediate2}
\leq 
\frac{q}{q-1}(P(H))^{1/r^{\prime}}
\E^{1/r}(
\sup_{(\tau,x)\in H_T}
|v_\tau(x)-\bar v_\tau(x)|
^{qr^{\prime}})
\end{equation}
\begin{equation}                                             \label{estimate0}
\leq \frac{2^{q-1}q}{q-1}(P(H))^{1/r^{\prime}}
V_T
\end{equation}
with 
$$
V_T:=\E^{1/r}
\sup_{(\tau,x)\in H_T}
|v_\tau(x)|
^{qr}
+\E^{1/r}
\sup_{(\tau,x)\in H_T}
|\bar v_\tau(x)|
^{qr},   
$$ 
for  $r>1$, $r^{\prime}=r/(r-1)$, 
provided $q>1$.  By Theorem \ref{theorem ex} 
\begin{equation}                                               \label{estimate3}
V_T\leq N\E^{1/r}(\cK^{qr}_{m,p,0}(T)
+\bar\cK^{qr}_{m,p,0}(T)). 
\end{equation}

We can estimate $P(H)$ as follows. Clearly, 
$$
P(H)\leq P(\sup_{(t,x)\in \bB_{\nu R}}|Y_t^{-1}(x)|>\nu'R)
+P(\sup_{(t,x)\in \bB_{\nu' R}}|Y_t(x)|>R)=:J_1+J_2.
$$
For $\hat{Y}_t(x)={Y}_t(x)-x$ by \eqref{lemma eq2} we have 
$$
J_2\leq P(\sup_{(t,x)\in \bB_{\nu' R}}|\hat{Y}_t(x)|>(1-\nu')R)
\leq NR^{d+1/2}e^{-\delta (1-\nu)^2R^2}.
$$
Also, we have 
$$
J_1
\leq
\sum_{l=0}^{\infty}
P(\exists(t,x)\in [0,T]\times (B_{2^{l+1}\nu' R}
\setminus B_{2^{l}\nu' R}):|Y_t(x)|\leq \nu R)
$$
$$
\leq
\sum_{l=0}^{\infty}
P(\sup_{(t,x)\in \bB_{2^{l+1}\nu' R}}|\hat{Y}_t(x)|
\geq (2^{l}\nu'-\nu)R).
$$
Using \eqref{lemma eq2} again gives
$$
J_1
\leq N\sum_{l=0}^{\infty}
e^{-\delta(2^{l}\nu'-\nu)^2R^2}(2^{l+1}\nu' R)^{d+1} \leq Ne^{-\delta R^2}        
$$
We can conclude that
\begin{equation}                                                                                \label{prob estimate}
P(H)\leq Ne^{-\delta R^2}, 
\end{equation}
where $N$ and $\delta$ are positive constants, 
depending only on $d$, $K$ and $T$. 

Combining this with \eqref{estimate0} and \eqref{estimate3} we can finish 
the proof of the theorem under the additional conditions. 

For general $\vartheta$ one applies the same arguments as in Remark \ref{remark poly}.
Finally \eqref{main estimate} for the case $q\in(0,1]$ follows easily from 
standard arguments using Lemma 3.2 from \cite{GK2003}. \qed

\section{An application - finite differences}                                                          \label{Section-findiff}
In this section we apply Theorem \ref{theorem cutoff} to present 
a numerical scheme approximating the initial value problem 
\eqref{SPDE general}-\eqref{SPDE general-ini}. 
We make use of the results of \cite{Gy} on the rate 
and acceleration of finite difference approximations 
in the spatial variable, which, together with 
a time discretization and a truncation - whose error can be 
estimated using Theorem \ref{theorem cutoff} - yields a 
fully implementable scheme. We shall carry out 
the steps of approximation in the following order: 
spatial discretization by finite differences, localization of the 
finite difference scheme, and discretization in time via implicit Euler's method. 
This of course requires an analysis of the Euler scheme, 
to present an error estimate for it, which does not depend on the spatial mesh size and the 
localization. Furthermore, in our full discretization scheme we shall incorporate  Richardson's extrapolation, which will allow us to improve the accuracy of the scheme in the spatial mesh size $h$.

First we introduce the finite difference approximation in the 
spatial variable for \eqref{SPDE general}-\eqref{SPDE general-ini}. 
To this end, let $\Lambda_1\subset\R^d$ be a finite set, containing the zero vector, satisfying the following natural condition: 
$\Lambda_0:=\Lambda_1\setminus\{0\}$ is not empty, and 
 if a subset $\Lambda'\subset\Lambda_0$ is linearly dependent, then it is linearly dependent over the rationals. 
Let $h>0$, and define the grid
$$
\G_h=\{h\sum_{i=1}^n\lambda_i:\lambda_i\in\Lambda_1\cup-\Lambda_1,n=1,2,\ldots\}. 
$$
Due to the assumption on $\Lambda_1$, $\G_h$ 
has only finitely many points in every ball around the origin in $\bR^d$. 
Define for $\lambda\in\Lambda_0\cup-\Lambda_0$, the finite difference operators
$$
\delta^h_{\lambda}\varphi(x)=\frac{1}{2h}(\varphi(x+h\lambda)-2\varphi(x)+\varphi(x-h\lambda)),
$$
and let $\delta^h_0$ stand for the identity operator. 
 To approximate the Cauchy problem \eqref{SPDE general}
-\eqref{SPDE general-ini}, for $h>0$ we consider the equation
\begin{equation}                                             \label{eq:finite diff}
du_t(x)=({L}^h_tu_t(x)+f_t(x))\,dt+\sum_{k=1}^{\infty}({M}^{h,k}_tu_t(x)+g_t^k(x))dw^k_t
\end{equation}
on $[0,T]\times\G_h$, with initial condition
\begin{equation}                                           \label{eq:finite diff-ini}
u_0(x)=\psi(x), 
\end{equation}
where ${L}^h$ and ${M}^{h,k}$ are difference operators 
of the form 
$$
{L}^h_t(x)=\sum_{\lambda,\kappa\in\Lambda_1}\fra_t^{\lambda\kappa}(x)\delta^h_{\lambda}\delta^h_{\kappa},
\quad 
{M}^{h,k}_t(x)=\sum_{\lambda\in\Lambda_1}
\frb^{\lambda,k}_t(x)\delta^h_{\lambda},\quad k=1,2,\ldots,  
$$
with some real-valued $\cP\otimes\cB(\bR^d)$-measurable 
$\fra_t^{\lambda\kappa}$ and $\frb^{\lambda,k}$ 
on $\Omega\times[0,T]$, such that 
\begin{equation}                                       \label{bounded}
|\fra_t^{\lambda\kappa}(x)|
\leq K \quad\text{and}\quad 
\sum_{k}|\frb_t^{\lambda,k}(x)|^2\leq K^2
\end{equation}
for all $\lambda,\kappa\in\Lambda_1$, $t\in[0,T]$, $x\in\bR^d$ 
and $\omega\in\Omega$, where $K$ is a constant. 
\begin{remark}                                \label{remark modification}
Here $\psi$, $f$ and $g$ are the same as in 
\eqref{SPDE general}-\eqref{SPDE general-ini} 
and we will assume that they satisfy Assumption 
\ref{assumption regularitydata} with $m>d/2$, $p=2$ 
and $\vartheta=0$. Thus by Sobolev's embedding of $W^m_2$ into $C_b$, 
the space of bounded continuous functions, for all 
$\omega$ we can find a continuous function of $x$ which is
equal to $\psi$ almost everywhere, and 
for each $t$ and $\omega$ we have 
continuous functions of $x$ which coincide with  $f_t$ and $g_t$
for almost every $x\in\bR^d$. Here and in the following 
we always take such continuous modifications if they exist, thus we always assume 
that $\psi$, $f_t$, and $g_t$ 
are continuous in $x$ for all $t$ (for $g=(g^k)_{k=1}^\infty$ 
this means, as usual, continuity as a function with values in $l_2$). 
In particular, terms like $f_t(x)$ in \eqref{eq:finite diff} make sense. 
 We note that for $m>d/2$ one can use 
Sobolev's theorem on embedding $W^m_2$ to $C_b$ 
to show also that 
if Assumption \ref{assumption regularitydata} holds with 
$m>d/2$, $p=2$ $\theta=0$, then 
$$
\sum_{x\in\G_h}|\psi(x)|^2h^d+\int_0^T\big(\sum_{x\in\G_h}|f_t(x)|^2h^d
+\sum_{x\in\G_h}\sum_k|g^k_t(x)|^2h^d\big)\,dt
$$
$$
\leq N\|\psi(x)\|^2_m+N\int_0^T\|f_t\|^2_m
+\sum_k\|g^k_t\|^2_m\,dt<\infty \,\text{(a.s.)}, 
$$
with a constant  $N=N(\Lambda_0,d)$, where $\|\cdot\|_m:=|\cdot|_{m,2.0}$. 
(See Lemma 4.2 in \cite{GK2010}.)
\end{remark}

Clearly, for $\varphi\in C_0^{\infty}(\bR^d)$ and $\lambda\neq0$ 
$$
\delta^h_{\lambda}\varphi(x)\to \partial_{\lambda}\varphi(x):=\lambda^iD_i\varphi(x)\quad \text{as $h\to0$}.
$$
Thus, in order to approximate $L$ and $M^k$ by $L^h$ and 
$M^{h,k}$, respectively, we need the following compatibility condition.
\begin{assumption}                                                                      \label{as2:compatibility}
For every $i,j=1,\ldots,d$, $k=1,\ldots$ 
and for $P\otimes dt\otimes dx$-almost all 
$(\omega,t,x)\in\Omega\times[0,T]\times\bR^d$ 
$$
a^{ij}=\sum_{\lambda,\kappa\in\Lambda_0}\fra^{\lambda\kappa}\lambda^i\kappa^j,\,\,
b^i=\sum_{\lambda\in\Lambda_0}(\fra^{0\lambda}+\fra^{\lambda0})\lambda^i,\,\,c=\fra^{00}, 
$$
$$
\sigma^{ik}=\sum_{\lambda\in\Lambda_0}\frb^{\lambda,k}\lambda^i,
\quad
 \mu^k=\frb^{0,k}.
$$
\end{assumption}

For each $x\in\G_h$ equation \eqref{eq:finite diff} is a stochastic differential 
equation (SDE), i.e., in general, \eqref{eq:finite diff}-\eqref{eq:finite diff-ini} is an infinite 
system of SDEs. 
To replace this with a finite system we make the coefficients, together  
with the free  and initial data, vanish outside of a large ball  
by multiplying them with a cutoff function $\zeta_R$, 
which satisfies the following condition. 
\begin{assumption}                                                      \label{assumption cutoff}
For an integer $m\geq0$ and a real number $R>0$ the function $\zeta_R$ is 
a continuous function with compact support on $\bR^d$, such that 
$\zeta(x)=1$  for $|x|\leq R$ and the derivatives of $\zeta_R$ up to order 
$m+1$ are continuous functions, bounded by a constant $C$. 
\end{assumption}

 In this way we replace  
\eqref{eq:finite diff}-\eqref{eq:finite diff-ini} with the system of SDEs
\begin{equation}                                                               \label{finite}
du_t(x)=(L^{h,R}_tu_t(x)+f^R_t(x))\,dt
+(M^{h,R,k}_tu_t(x)+g_t^{R,k}(x))\,dw_t^k, \quad t\in[0,T],
\end{equation}
with initial condition 
\begin{equation}                                                                 \label{finite ini}
u_0(x)=\psi^R(x), 
\end{equation}
for $x\in\G_h\cap\rm{supp}\,\zeta_R$, where $\rm{supp}\,\zeta_R$ 
is the support of $\zeta_R$,  
\begin{equation}                                                 \label{free R}
(\psi^R,f_t^R, g_t^{R,k}):=(\zeta_R\psi,\zeta_Rf_t,\zeta_R g_t^{k})
\end{equation}
and 
$$
{L}^{h,R}_t:=\sum_{\lambda,\kappa\in\Lambda_1}\fra_t^{\lambda\kappa,R}\delta^h_{\lambda}\delta^h_{\kappa},
\quad M^{h,R,k}_t:=\sum_{\lambda\in\Lambda_1}
\frb^{\lambda,R,k}_t\delta^h_{\lambda},\quad k=1,2,\ldots,  
$$
with 
\begin{equation}\label{coeffcutoff}
\fra^{\lambda\kappa,R}
:=\zeta_R^2\fra^{\lambda\kappa}
\quad \text{for $\lambda,\kappa\in\Lambda_0$}, 
\end{equation}
\begin{equation}\label{coeffcutoff2}
(\fra^{0\kappa,R}, \fra^{\lambda0,R}, \frb^{\lambda,R,k})
:=(\zeta_R\fra^{0\kappa},\zeta_R \fra^{\lambda0},\zeta_R \frb^{\lambda,k})
\quad \text{for $\lambda,\kappa\in\Lambda_1$, $k\geq1$}.
\end{equation}

At this point our approximation is a finite dimensional (affine) linear SDE,  whose coefficients are bounded by $K$ owing to \eqref{bounded}, and furthermore, by virtue of Remark \ref{remark modification}, for each $x$, $f^R(x)$ and $g^R(x)$ are square integrable in time under Assumption \ref{assumption cutoff} and \ref{assumption free} below with $m>d/2$.

Hence \eqref{finite}-\eqref{finite ini} has a unique solution 
$$
\{u^{h,R}_t(x): x\in\G_h\cap\rm{supp\,\zeta_R}\}_{t\in[0,T]},  
$$
by virtue of a well-known 
theorem of It\^o on finite dimensional SDEs with Lipschitz continuous coefficients.  
The approximation of such equations are well studied, various time-discretization  
methods can be used, each of them with their own advantages and  disadvantages. Here we chose the implicit Euler method, formulated as follows.

We take a mesh-size $\tau=T/n$ for an integer $n\geq1$, and 
approximate \eqref{eq:finite diff}-\eqref{eq:finite diff-ini} by the equations
\begin{equation}                                        \label{eq:discrete time-ini}
v_0(x)=\psi^R(x),  
\end{equation}
$$
v_i(x)=v_{i-1}(x)+(L^{h,R}_{\tau i}v_i(x)+f_{\tau(i-1)}^R(x))\tau
$$
\begin{equation}                                            \label{eq:discrete time}
+\sum_{k=1}^{\infty}
(M^{h,R,k}_{\tau(i-1)}v_{i-1}(x)+g^{R,k}_{\tau(i-1)}(x))\xi_{i}^k, 
\quad i=1,2,\ldots,n,
\end{equation}
for $x\in\G_h\cup\rm{supp}\,\eta_R$,  where $\xi_i^k=w^k_{i\tau}-w^k_{(i-1)\tau}$. 

\begin{remark}
In many applications, including the Zakai equation for nonlinear filtering, the driving noise is finite dimensional. If this is not the case, one needs another level of approximation, at which the infinite sum in \eqref{eq:discrete time} is replaced by its first $m$ terms. We shall not discuss this here.
\end{remark}

\begin{remark}
As mentioned before, Euler approximations for SDEs are very well studied. 
Therefore, while it is far from immediate that the error is of the desired order, 
independently of $h$ and $R$, 
the implementation of the scheme goes as usual, see e.g. \cite{Klo} and its references.
\end{remark}

To prove to solvability of the fully discretized equation \eqref{eq:discrete time-ini}-\eqref{eq:discrete time} and estimate its error from the true solution of \eqref{SPDE general}-\eqref{SPDE general-ini} on the space-time grid we pose the following assumptions. As for the following we confine ourselves to the $L_2$-scale, 
without weights, we use the shorthand 
notation $\|\cdot\|_m=|\cdot|_{W^m_{2,0}}$, $\|\cdot\|=\|\cdot\|_0$. 

\begin{assumption}                                             \label{as2:parabolicity}
For all $(\omega,t,x)\in\Omega\times[0,T]\times\bR^d$ 
$$
 \sum_{\lambda,\kappa\in\Lambda_0}(2\fra^{\lambda\kappa}-\frb^{\lambda,k}\frb^{\kappa,k})z^{\lambda}z^{\kappa}\geq 0
$$
for all $z=(z^{\lambda})_{\lambda\in\Lambda_0}$, 
$z^{\lambda}\in\bR$.  
\end{assumption}

In the following assumptions $m$ and $l$ are nonnegative integers, 
as before, and will be more specified in the theorems below.
\begin{assumption}                                                   \label{as2:regularity}
The derivatives in $x$ of $\fra^{\lambda\kappa}$ up 
to order $\max(m,2)$ are $\mathcal P\otimes\mathcal B(\bR^d)$-measurable 
functions, bounded by $K$ for all $\lambda,\kappa\in\Lambda_1$. 
The derivatives in $x$ of   $\frb^{\lambda}
=(\frb^{\lambda r})_{r=1}^{\infty}$  
up to order $m+1$ are 
$\mathcal P\otimes\mathcal B(\bR^d)$-measurable 
$l_2$-valued functions, bounded by $K$, for all $\lambda\in\Lambda_1$.
\end{assumption}

\begin{assumption}                       \label{assumption free}
The initial value, 
$\psi$ is an $\cF_0$-measurable random variable 
with values in $W^m_{2}$. 
The free data, $ f_t $ and $ g_t =(g^k)_{k=1}^{\infty}$ are 
predictable 
processes with values in $W^m_{2}$ and 
$W^{m+1}_{2}(l_2)$, respectively, 
such that almost surely
\begin{equation}			                                    \label{K_m}	
\mathcal K_{m}^2:=\|\psi\|_{m}^2+\int_0^T
\big(\|f_t\|^2_{m}+\|g_t\|^2_{m+1 }\big)\,dt<\infty.
\end{equation}
\end{assumption}

\begin{assumption}                                                                  \label{assumption holder}
There exists a constant $H$  such that 
$$
\E\|f_t-f_s\|^2_{l}+\E\|g_s-g_t\|^2_{l+1}\leq H|t-s|, 
\quad 
\E\|f_t\|^2_{l+1}+\E\|g_t\|^2_{l+2}\leq H
$$
for all $s,t\in[0,T]$, and   
$$
|D^{\alpha}(\fra^{\lambda\kappa}_t(x)-\fra^{\lambda\kappa}_s(x))|^2\leq H|t-s|, 
\quad 
\sum_{k}|D^{\beta}(\frb^{\lambda,k}_t(x)-\frb^{\lambda,k}_s(x))|^2\leq H|t-s|, 
$$
for all $\omega\in\Omega$, $x\in\R^d$, $s,t\in[0,T]$ and multi-indices $\alpha$ 
and $\beta$ with $|\alpha|\leq l$ and $|\beta|\leq l+1$. 
\end{assumption}

\begin{remark}                                                              \label{remark parabolicity1}
If Assumptions \ref{as2:compatibility}  
and \ref{as2:parabolicity} hold then 
$$
(2a^{ij}-\mu^{ir}\mu^{jr})z^iz^j=
(2\sum_{\lambda,\kappa\in\Lambda_0}\fra^{\lambda\kappa}\lambda^{i}\kappa^{j}
-\sum_{\lambda\in\Lambda_0}\frb^{\lambda,k}\lambda^i
\sum_{\kappa\in\Lambda_0}\frb^{\kappa,k}\kappa^j)z^iz^j
$$
\begin{equation*}                                                  
=\sum_{\lambda,\kappa\in\Lambda_0}
(2\fra^{\lambda\kappa}
-\frb^{\lambda,k}\frb^{\kappa,k})
(\lambda^{i}z^i)(\kappa^{j}z^j)\geq0  
\end{equation*}
for all $z=(z^1,...,z^d)\in\bR^d$, i.e.,
Assumption \ref{assumption parabolicity} also holds. 
Clearly, Assumptions \ref{as2:compatibility} and \ref{as2:regularity} 
imply Assumption \ref{assumption regularitycoeff} (a)-(b). 
Notice that Assumption \ref{assumption free} is the same 
as Assumption \ref{assumption regularitydata} with $p=2$ and 
$\vartheta=0$. 
Thus if Assumptions \ref{assumption regularitycoeff} (c), and 
\ref{as2:compatibility} through \ref{assumption free} hold with $m>2+d/2$, 
then by virtue of Corollary \ref{corollary ex} equation \eqref{SPDE general}  
with initial condition \eqref{SPDE general-ini} has a unique classical 
solution 
$$
u=\{u_t(x):t\in[0,T], x\in\bR^d\}. 
$$
\end{remark}

Now we are in the position to formulate 
the first main theorem of this section, with the notation $\G_h^R=\G_h\cap B_R$.
\begin{theorem}                                               \label{theorem main1}
Let $l>d/2$ be an integer. Let Assumptions 
\ref{as2:compatibility} through \ref{as2:regularity} hold with $m\geq 4+l$, 
and let Assumptions \ref{assumption regularitycoeff} (c) and 
Assumption \ref{assumption holder} hold with $m\geq2+l$ and with $l+1$, 
respectively.  Then if $\tau$ is sufficiently small, 
then for any $h>0$, $R>1$  
the system of equations
\eqref{eq:discrete time-ini}-\eqref{eq:discrete time} 
has a unique solution $(v^{R,h,\tau}_i)_{i=0}^n$.  Moreover, 
for any $\nu\in(0,1)$, $q>1$  
we have 
$$
\E\max_{i=0,\ldots,n}\max_{x\in\G^{\nu R}_h}
|u_{\tau i}(x)-v^{h,R,\tau}_i(x)|^2
$$
\begin{equation}\label{eq:full discret error}
\leq N_1e^{-\delta R^2}\E^{1/q}\cK_{l+2}^{2q}
+N_2(h^{4}+\tau)(1+\E\cK_m^{2}),
\end{equation}
with constants $N_1$ and $\delta>0$ depending only on $K$, $d$, $T$, $C$ 
$q$, $\nu$ and  $\Lambda_0$, 
and a constant $N_2=N_2(K,T,d,C,H,\Lambda_0)$. 
\end{theorem}
As mentioned above we want to have approximations 
with higher order accuracy in $h$ by extrapolating from 
$v^{h,R,\tau}$. Let us recall the method of Richardson's extrapolation. 
This technique, first introduced in \cite{R}, 
allows one to accelerate the rate of convergence 
by appropriately mixing approximations with different mesh sizes, 
given that a power expansion of the error in terms of the mesh sizes
is available. We shall use this, based on results of \cite{Gy}, 
to obtain higher order approximations with respect to the spatial mesh size 
$h$. To formulate the extrapolation, 
let $r\geq0$, $V$ be the $(r+1)\times(r+1)$ 
Vandermonde matrix $V^{ij}=(4^{-(i-1)(j-1)})$,
\begin{equation}                                             \label{c}
(c_0,c_1,\ldots,c_r):=(1,0,\ldots,0)V^{-1},
\end{equation}
and define
\begin{equation}                                                                         \label{extrapol}
\bar v^{h,R,\tau}:=\sum_{i=0}^rc_i{v}^{h/2^i,R,\tau}, 
\end{equation}
where ${v}^{h/2^i,R,\tau}$ denotes the solution of 
\eqref{eq:discrete time-ini}-\eqref{eq:discrete time} with $h/2^i$ 
in place of $h$.
As we shall see, even by mixing only two approximations with 
different mesh sizes, that is, setting $r=1$, the extrapolation increases the order of accuracy in $h$ from 2 to 4.

The second main result of this section is the following.
\begin{theorem}                                              \label{theorem main2}
In additions to the assumptions of Theorem \ref{theorem main1} 
let Assumptions \ref{assumption cutoff}, \ref{as2:regularity} and 
\ref{assumption free} hold with $m\geq 4r+4+l$. 
Then for 
the extrapolation $\bar v^{h,R,\tau}$ 
we have 
$$
\E\max_{i=0,\ldots,n}\max_{x\in\G_h^{\nu R}}
|u_{\tau i}(x)-\bar v^{h,R,\tau}_i(x)|^2
$$
\begin{equation}                                                           \label{eq:full discret accelerated}
\leq N_1e^{-\delta R^2}\E^{1/q}\cK^{2q}_{2+l}
+N_2(h^{2(2r+2)}+\tau)(1+\E\cK_m^{2})
\end{equation}
for any $\nu\in(0,1)$ and $q>1$,  
with constants $N_1$ and $\delta>0$, depending only 
on $K$, $d$, $T$, $C$, $\nu$, $q$ and $\Lambda_0$, 
and a constant $N_2=N_2(K,T,d,C,H,r,\Lambda_0)$.  
\end{theorem}

These theorems will be proved by using Theorem 
\ref{theorem cutoff}, some results from \cite{Gy}, summarized 
below in Theorem \ref{theorem Gy}, and the error estimate for the 
time-discretization, established in Theorem \ref{theorem time error} below.

\begin{example}
Consider the equation
$$
du_t(x)=\sin^2(x)D^2 u_t(x)dt+\sin(x)D u_t(x)dw_t
$$
for $(t,x)\in[0,1]\times\R$, where $(w_t)_{t\in[0,1]}$ 
is a 1-dimensional Wiener process, with the initial condition
$$
u_0(x)=(1+x^2)^{-1}.
$$
The choice of localizing function $\zeta_{R}$ is quite arbitrary, 
for the sake of concreteness we take
$
\zeta_R(x):=f(x+2+R)-f(x-2-R),\,
$
where
$$
f(x):=\frac{2}{\pi}\arctan e^{x/(1-x^2)}\quad \text{for $|x|<1$, and 
$f(x):={\mathbf 1}_{[1,\infty)}(x)$ for $|x|\geq1$},
$$
while noting that in practice a simple mollified indicator of $[-R,R]$ 
may be more favourable. Notice that $\zeta_R(x)=1$ for $|x|\leq R$ and 
$\rm{supp}\,\zeta_R=[-3-R,3+R]$.

For an integers $j\geq1$ and $n\geq1$ we set $h=R/(10j)$ and $\tau=1/n$. 
To use the extrapolation with $r=1$ in \eqref{extrapol}, we need to solve 
two discrete equations with spatial mesh sizes $\bar h=h,h/2$, 
and mix them according to \eqref{extrapol}, 
where one can check that the coefficients  are $c_0=-1/3, c_1=4/3$. 
Following the steps outlined above, the discrete equation we arrive at is
$$
u^{R,\bar h,\tau}_{i}(k\bar h)=\fra(k\bar h)(\delta^{\bar h}\delta^{\bar h}u^{R,\bar h}_i)(k\bar h)\tau+\frb(k\bar h)(\delta^{\bar h}u^{R,\bar h}_{i-1})(k\bar h)(w_{\tau i}-w_{\tau(i-1)})
$$
for $i=1,2,\ldots,n$ and 
$k=0, \pm1,\ldots, \pm \lceil(3+R)/\bar h\rceil$, 
with the initial values
$$
u^{R,\bar h,\tau}_{0}(k\bar h)=(1+k^2\bar h^2)^{-1}\zeta_R(k\bar h),
$$
where $\fra(x)=\zeta^2_{R}(x)\sin^2(x)$, 
$\frb(x)=\zeta_{R}(x)\sin(x)$, and 
$$
\delta^{\bar h}\phi(x)=(2\bar h)^{-1}[\phi(x+\bar h)-2\phi(x)+\phi(x-\bar h)].
$$
For each $\bar h$ one can solve the above equation recursively in $i$. 

Taking the resulting solutions $u^{R,h,\tau}$ 
and $u^{R,h/2,\tau}$, and setting 
$$
v^{R,h,\tau}=(4/3)u^{R,h/2,\tau}-(1/3)u^{R,h,\tau},
$$ 
we can conclude that the error 
$\E\max_{i,k}|u_{i\tau}(kh)-v_i^{R,h,\tau}(kh)|^2$, 
where $i$ runs over $0,1,\ldots, n$ 
and $k$ runs over $0\pm1,\ldots,\pm0.9R/h$, 
is of order $e^{-\delta R^2}+h^8+\tau$.
\end{example}

Before summarising some results 
from \cite{Gy} on finite difference operators $L^h$, $M^h$ and stochastic 
finite difference schemes in spatial variables we need to make 
an important remark.

\begin{remark}                                                   \label{remark important}
The concept of a solution of \eqref{eq:finite diff}-\eqref{eq:finite diff-ini}, 
as a process with values in $l_{2,h}$,  
the space of functions $\phi:\G_h\rightarrow\R$ 
with finite norm $\|\phi\|_{l_{2,h}}^2=\sum_{x\in\G_h}|\phi(x)|^2$, is straightforward. 
One can, however, also consider \eqref{eq:finite diff}-\eqref{eq:finite diff-ini} 
on the whole space, that is, for $(t,x)\in H_T$. 
 Namely, when  Assumptions 
\ref{as2:regularity} and \ref{assumption free} hold, 
then we can look for an $\cF_t$-adapted 
$L_2$-valued solution $(u_t^h)_{t\in[0,T]}$ such that 
almost surely for every $t\in[0,T]$
\begin{equation}                                     \label{SDE2}
u_t^h=\psi+\int_0^t(L^h_su_s^h+f_s)\,ds
+\sum_{k=1}^{\infty}\int_0^t(M^{h,k}_su_s^h+g^{k}_s)\,dw^k_s
\end{equation}
in the Hilbert space $L_2$, where the first integral is understood 
as Bochner integral of $L_2$-valued 
functions, the stochastic integrals are understood as 
It\^o integrals of $L_2$-valued processes, 
and the convergence of their infinite sum is understood in probability, 
uniformly in $t\in[0,T]$. 
Thus, by a well-known theorem on SDEs in Hilbert spaces 
with Lipschitz continuous 
coefficients,  equation \eqref{SDE2} has a unique 
$L_2$-valued continuous $\cF_t$-adapted solution $u^h=(u^h_t)_{t\in[0,T]}$.  
We refer to such a solution as an {\it $L_2$-valued solution} to 
\eqref{eq:finite diff}-\eqref{eq:finite diff-ini}.  
We can view equation \eqref{SDE2} 
also as an SDE in the Hilbert space $W^m_2$, and by the same theorem on 
existence and uniqueness of solution to SDEs in Hilbert spaces, we get a unique 
$W^m_2$-valued continuous $\cF_t$-adapted solution to it. 
Consequently, 
if Assumptions \ref{as2:regularity} and \ref{assumption free} 
hold with $m\geq0$, then 
$u=(u^h_t)_{t\in[0,T]}$, the $L_2$-valued solution to 
\eqref{eq:finite diff}-\eqref{eq:finite diff-ini} is a 
$W^m_2$-valued continuous $\cF_t$-adapted 
process. Also, it is straightforward to see that these 
two concepts of solutions are ``compatible'' in the sense 
that if $m>d/2$, then the restriction of the $L_2$-valued solution 
to $\G_h$ solves \eqref{eq:finite diff}-\eqref{eq:finite diff-ini} 
as an $l_{2,h}$-valued process. 
Note that \eqref{finite}-\eqref{finite ini} is a 
special case of the class of equations 
of the form \eqref{eq:finite diff}-\eqref{eq:finite diff-ini}, 
and so the above discussion applies to it as well.

The analogous concepts will be used for solutions 
of \eqref{eq:discrete time-ini}-\eqref{eq:discrete time}.
Namely, a sequence of $L_2$-valued random variables 
$(v^{h,R}_i)_{i=0}^n$ is called 
an $L_2$-valued solution to 
\eqref{eq:discrete time-ini}-\eqref{eq:discrete time} if $v^{h,R}_i$ 
is $\cF_{i\tau}$-measurable for $i=0,1,2,...,n$, 
and the equalities hold for almost all $x\in\bR^d$, for 
almost all $\omega\in\Omega$. If 
$(v^{h,R}_i)_{i=0}^n$ is  
an $L_2$-valued solution to 
\eqref{eq:discrete time-ini}-\eqref{eq:discrete time} such that 
$v^{h,R}_i\in W^m_2$ (a.s.) for $m>d/2$ for each $i$, then it is easy to 
see that the restriction of the continuous version of $v^{h,R}_i$ to 
$\G_h\cap\rm{supp}\,\zeta_R$ for each $i$ gives a solution 
$\{v_i(x): x\in\G_h\cap\rm{supp}\,\zeta_R, i=0,1,...,n\}$ to 
\eqref{eq:discrete time-ini}-\eqref{eq:discrete time}. 
\end{remark}

\begin{theorem}                                                                 \label{theorem Gy}
Let Assumptions \ref{as2:parabolicity}, 
\ref{as2:regularity} and \ref{assumption free} 
hold with an integer $m\geq0$. Then
\begin{enumerate}[(a)]
\item For any $\phi\in W^m_p$ and $|\gamma|\leq m$ 
$$
2(D^{\gamma}\phi,D^{\gamma}{L}^h_t\phi)
+\sum_k\|D^{\gamma}{M}^{h,k}_t\phi\|^2\leq N\|\phi\|_m^2
$$
for all $\omega\in\Omega$ and $t\in[0,T]$
with a constant $N=N(K,m,d, \Lambda_0)$;
\item There is a unique $L_2$-valued  
solution ${u}^h$ of \eqref{eq:finite diff}-\eqref{eq:finite diff-ini}. It is 
a $W^m_2$-valued process with probability one, and for any $q>0$
\begin{equation*}                                              \label{sup}
\E \sup_{t\leq T}\|{u}^h_t\|_m^q\leq N\E\cK^q_m
\end{equation*}
with a constant $N=N(K,m,d,\Lambda_0,q,T)$;
\item If for an integer $r\geq0$ 
Assumptions \ref{as2:compatibility} through \ref{assumption free} 
with $m>4r+4+d/2$ hold, then  for any $q>0$
$$
\E \sup_{t\leq T}\max_{x\in\G_h}|u_t(x)-\bar u^h(x)|^q\leq 
Nh^{q(2r+2)}\E\cK^q_m,
$$
where $u$ is the classical solution to 
\eqref{SPDE general}-\eqref{SPDE general-ini}, 
$\bar u^h=\sum_{i=0}^ru^{h/2^i}$, and $N$ 
is a constant depending only on 
$K$, $d$, $T$, $q$, $\Lambda_0$ and $m$.
\end{enumerate}
\end{theorem}

As discussed above, under  Assumptions 
\ref{as2:regularity} and \ref{assumption free} we have  
a unique $L_2$-valued 
solution $u^{h,R}=(u^{h,R}_t)_{t\in[0,T]}$ to \eqref{finite}-\eqref{finite ini}, 
and it is a continuous $W^m_2$-valued process.

\begin{theorem}                                                       \label{theorem time error}
(i) Let Assumptions \ref{assumption cutoff} through 
 \ref{assumption free} hold with $m\geq0$.  
Then for sufficiently small $\tau$ there exists for all $h$ and $R>0$ 
a unique 
$L_2$-valued solution $v^{h,R,\tau}$ to 
\eqref{eq:discrete time-ini}-\eqref{eq:discrete time} such that 
$v_i^{h,R,\tau}\in W^m_2$ for every $i=0,1,...,n$ and 
$\omega\in\Omega$. 
\newline
(ii) If Assumption \ref{assumption holder} 
holds with some integer $l\geq0$ and Assumptions \ref{as2:compatibility}   
through 
 \ref{assumption free}
hold with $m=l+3$, then 
\begin{equation}                                 \label{gyenge becsles}
\max_{i\leq n}\E \|u^{h,R}_{\tau i}-v^{h,R,\tau}_i\|_l^2
\leq N\tau(1+\E\cK_{m}^{2}),
\end{equation}
with a constant 
$N=N(K,C,H,d,T,l,\Lambda_0)$.  

\noindent
(iii) Let $l\geq 0$ be an integer. If  Assumption \ref{assumption holder} 
holds with $l+1$ in place of $l$, and Assumptions \ref{as2:compatibility}   
through 
 \ref{assumption free}
hold with $m=l+4$, then 
\begin{equation} 						\label{erosebb becsles}
\E \max_{i\leq n}\|u^{h,R}_{\tau i}-v^{h,R,\tau}_i\|_l^2
\leq N\tau(1+\E\cK_{m}^{2})
\end{equation}
with a constant 
$N=N(K,C,H,d,T,l,\Lambda_0)$.  
\end{theorem}

\begin{proof} 
To prove solvability of the system of equations 
\eqref{eq:discrete time-ini}-\eqref{eq:discrete time} we rewrite 
\eqref{eq:discrete time} in the form 
\begin{equation}                                                    \label{inverse}
(I-\tau L^{h,R}_{\tau i})v_i=v_{i-1}+\tau f_{\tau(i-1)}
+\sum_{k=1}^{\infty}
(M^{h,R,k}_{\tau(i-1)}v_{i-1}+g^{R,k}_{\tau(i-1)})\xi_{i}^k, 
\quad i=1,2,\ldots,n,
\end{equation}
where $I$ denotes the identity operator. We are going to show by induction on $j\leq n$ 
that for sufficiently small $\tau$  for each $j$ there is a sequence of 
$W^m_2$-valued random variables 
$(v_i)_{i=0}^{j}$, such that $v^i$ is $\cF_{i\tau}$-measurable, 
$v_0=\psi^{R}$ and 
\eqref{inverse} holds for $1\leq i\leq j$. For $j=0$ there is nothing to prove. 
Let $j\geq1$ and assume that our statement holds for $j-1$. Consider the equation 
\begin{equation}                         \label{elliptic}
\bD v=X, 
\end{equation}
where 
$$
\bD:=I-\tau L^{h,R}_{\tau j}, \quad 
X:=v_{j-1}+\tau f_{\tau(j-1)}
+\sum_{k=1}^{\infty}
(M^{h,R,k}_{\tau(j-1)}v_{i-1}+g^{R,k}_{\tau(j-1)})\xi_{j}^k.  
$$
In the following we take $\kappa$ to be either $0$ or $m$. It easy to see that $\bD$ is a bounded linear operator from $W^\kappa_2$ into 
$W^\kappa_2$, 
for each  $\omega\in\Omega$ and $\tau$. 
Let us define the norm $|\cdot|_\kappa$ in $W^\kappa_2$ by 
$$
|\varphi|^2_\kappa=\|(I-\Delta)^{\kappa/2}\phi\|^2
=\sum_{\gamma:|\gamma|\leq\kappa}C_{\gamma}\|D^{\gamma}\phi\|^2, \quad 
\Delta:=\sum_{i=1}^dD_i^2,
$$
where 
$C_{\gamma}$ is a positive integer for 
each multi-index $\gamma$, $|\gamma|\leq \kappa$. 
Thus $|\cdot|_\kappa$ is a Hilbert norm which is equivalent to 
to $\|\cdot\|_\kappa$. 
We denote the corresponding inner product in $W^\kappa_2$  
by $(\,,\,)_\kappa$. 
By virtue of Theorem \ref{theorem Gy} (a)  for all 
$\omega\in\Omega$ and $\tau$ we have 
\begin{equation*}                                                      \label{coercive}
(v,Dv)_\kappa=|v|_\kappa-\tau(L^{h,R}_{\tau j}v,v)_\kappa\geq |v|^2_\kappa-\tau N|v|^2_\kappa, 
\quad\text{for all $v\in W^\kappa_2$}
\end{equation*}
where the dependence of $N$ is as in the theorem, in particular, it is independent of $h,R$. Consequently,
 for $\tau<1/N$ we have 
$$
(v,\bD v)_\kappa\geq\delta|v|_\kappa^2\quad\text{for all $v\in W^\kappa_2$}, 
\quad\omega\in\Omega,
$$
where $\delta=1-\tau N>0$. Hence by the Lax-Milgram lemma for every $\omega\in\Omega$ there is 
a unique $v=v_\kappa\in W^\kappa_2$ such that 
$$
(\bD v_\kappa,\varphi)_\kappa=(X,\varphi)_\kappa\quad \text{for all $\varphi\in C_0^{\infty}(\bR^d)$}. 
$$
Since $(Y,\varphi)_\kappa=(Y,(I-\Delta)^\kappa\varphi)$ 
for all $Y\in W^\kappa_2$ and $\varphi\in C_0^{\infty}(\bR^d)$, we have 
$$
(\bD v_\kappa,(I-\Delta)^\kappa\varphi)=(X,(I-\Delta)^\kappa\varphi)
\quad \text{for all $\varphi\in C_0^{\infty}(\bR^d)$}.
$$
Hence, taking into account that $\{(I-\Delta)^m\varphi):\varphi\in C_0^{\infty}(\bR^d)\}$ 
is dense in $W^0_2=L_2$, we get that $v_m$ solves \eqref{elliptic} in $L_2$ as well, so by uniqueness, $v_m=v_0$. This means 
\eqref{elliptic} has a unique solution $v\in L_2$ for every $\omega\in\Omega$, and 
$v\in W^m_2$ for every $\omega\in\Omega$. Since $X$ and $\bD \phi$ are 
 $W^m_2$-valued 
$\cF_{j\tau}$-measurable random variables for every $\varphi\in W^m_2$, the unique 
solution $v\in W^m_2$ to \eqref{elliptic} is also $\cF_{j\tau}$-measurable. 
This finishes the induction, and the proof of statement (i) of the theorem.

For parts (ii) and (iii) $\E\cK_{m}^{2}<\infty$ may and will be assumed. 
Hence by Theorem \ref{theorem Gy} (b), 
 \begin{equation}                                                 \label{veges2}
 \E\sup_{t\in[0,T]}|u^{h,R}_t|^{2}_{\varrho}\leq N \E\cK^2_{\varrho}<\infty 
 \quad \text{for $\varrho=0,1,2,...,m$}.  
 \end{equation}
As Assumption \ref{assumption holder} also holds, we have 
 \begin{equation}                                                          \label{condition}
\E\|\psi\|_{l+\kappa}^2+\max_{i\leq n}\E\|f_{t_i}\|_{l+\kappa}^2
+\max_{i\leq n}\sum_k\E|g^{k}_{t_i}\|^2_{l+\kappa}<\infty
\end{equation}
with $\kappa=0$ in part (ii) and $\kappa=1$ in part (iii), and hence, with the same $\kappa$,
\begin{equation}                                                      \label{veges1}
\E\|v_i^{h,R,\tau}\|_{l+\kappa}^2<\infty \quad\text{for every $i=0,1,...,n$.}
\end{equation}

To start the proof of (ii), let us fix a multi-index $\gamma$ with length
$|\gamma|=:\varrho\leq l$. 
From equations  \eqref{eq:discrete time-ini}-\eqref{eq:discrete time} and 
\eqref{eq:finite diff}-\eqref{eq:finite diff-ini}, 
we get that the error $e_i:={u}^{h,R}_{\tau i}-{v}_i^{h,R,\tau}$ 
is a $W^m_2$-valued $\mathcal{F}_{\tau i}$-measurable random variable, 
$i=0,\ldots,n$, and $D^{\gamma}e_i$ is the $L_2$-valued solution of the equation
\begin{align*}
D^{\gamma}e_i=&D^{\gamma}e_{i-1}                                                          \nonumber\\
&+D^{\gamma}L^{h,R}_{\tau i}e_i\tau+\int_{\tau(i-1)}^{\tau i}
D^{\gamma}F_s\,ds            \nonumber\\
&+D^{\gamma}M^{h,R,k}_{\tau(i-1)}e_{i-1}\xi_{i}^k
+\int_{\tau(i-1)}^{\tau i}D^{\gamma}G^k_s\,dw^k_s                                       
\end{align*}
for $i=1,\ldots,n$, with zero initial condition, where 
$$
F_t:=L^{h,R}_tu^{h,R}_t-L^{h,R}_{\kappa_2(t)}
u^{h,R}_{\kappa_2(t)}+f^{R}_t-f^{R}_{\kappa_1(t)},
\quad
\kappa_2(t):=\kappa_2^n(t):=(\lfloor nt\rfloor+1)/n,
$$
$$
G_t^k:=M^{h,k}_t u^{h,R}_t-M^{h,R,k}_{\kappa_1(t)}u^{h,R}_{\kappa_1(t)}
+{g}^{R,k}_t- g^k_{\kappa_1(t)},
\quad 
\kappa_1(t):=\kappa_1^n(t):=\lfloor nt\rfloor/n. 
$$
To ease notation we set 
$$
\frL_i:=L^{h,R}_{\tau i},\quad \frM_i^k:=M^{h,R,k}_{\tau i},
\quad
\frF_i:=\int_{\tau(i-1)}^{\tau i}D^{\gamma}F_s\,ds,
\quad
\frG_i:=\int_{\tau(i-1)}^{\tau i}D^{\gamma}G^k_s\,dw^k_s,
$$
and by using the simple identity 
$\|b\|^2-\|a\|^2=2(b,b-a)-\|b-a\|^2$ with $b:=D^{\gamma}e_i$ and 
$a:=D^{\gamma}e_{i-1}$,
we get 
\begin{align}
\|D^{\gamma}&e_i\|^2-\|D^{\gamma}e_{i-1}\|^2                                              \nonumber\\
&=2(D^{\gamma}e_i,D^{\gamma}\frL_i e_i\tau+\frF_i)
+2(D^{\gamma}e_{i},
D^{\gamma}\frM_{i-1}^ke_{i-1}\xi_i^k+\frG_i)                                                 
-\|D^{\gamma}e_i-D^{\gamma}e_{i-1}\|^2                                                     \nonumber\\                   
&=2(D^{\gamma}e_i,D^{\gamma}\frL_i e_i\tau+\frF_i)
+2(D^{\gamma}e_{i-1},
D^{\gamma}\frM_{i-1}^ke_{i-1}\xi_i^k+\frG_i)                                                 \nonumber\\
&\ \ 
+2(D^{\gamma}e_i-D^{\gamma}e_{i-1},
D^{\gamma}\frM_{i-1}^ke_{i-1}\xi_i^k+\frG_i)
-\|D^{\gamma}e_i-D^{\gamma}e_{i-1}\|^2                                                            \nonumber\\
&=2(D^{\gamma}e_i,D^{\gamma}\frL_{i}ie_i\tau+\frF_i)         
+2(D^{\gamma}e_{i-1},
D^{\gamma}\frM_{i-1}^ke_{i-1}\xi_i^k+\frG_i)                                                    \nonumber\\
&+\|D^{\gamma}\frM_{i-1}^ke_{i-1}\xi_i^k+\frG_i\|^2                                     
-\|D^{\gamma}\frL_{i}e_i\tau+\frF_i\|^2                                                              \nonumber\\
&\leq I_i^{(1)}+I_i^{(2)}+I_i^{(3)}+I_i^{(4)}+I_i^{(5)}+I_i^{(6)}                            \label{error estimate 1}                                                                                                            
\end{align}
with 
\begin{align*}
I_i^{(1)}:=&2\tau(D^{\gamma}e_i,D^{\gamma}\frL_{i}e_i)                                          \nonumber\\
I_i^{(2)}:=&2(D^{\gamma}e_i,\frF_i)                                                                           \nonumber\\
I_i^{(3)}:=&2(D^{\gamma}e_{i-1},D^{\gamma}\frM_{i-1}^ke_{i-1}\xi_i^k+\frG_i)         \nonumber\\
I_i^{(4)}:=&\|D^{\gamma}\frM_{i-1}^ke_{i-1}\xi_i^k\|^2                                                \nonumber\\
I_i^{(5)}:=&2(D^{\gamma}\frM_{i-1}^ke_{i-1}\xi_i^k,  \frG_i)                                        \nonumber\\
I_i^{(6)}:=&\|\frG_i\|^2.                                                             \nonumber\\
\end{align*}                                                                                                               
By  the Young and Jensen  inequalities, 
and basic properties of stochastic It\^o integrals we have 
\begin{equation}                                                                                    \label{est rev0}
I_i^{(2)}\leq \tau\|D^{\gamma}e_i\|^2+\tau^{-1}\|\frF_i\|^2\leq 
\tau\|D^{\gamma}e_i\|^2 +\int_{\tau(i-1)}^{\tau i}\|F_s\|^2_{\varrho}\,ds,
\quad
\E I_i^{(3)}=0,
\end{equation}
\begin{equation}                                                          \label{2.9.2}
 \E I_i^{(4)}
=\tau\E\sum_k\|D^{\gamma}\frM_{i-1}^ke_{i-1}\|^2, 
\end{equation}
\begin{equation}                                       \label{1.9.2}
\E I_i^{(6)}=
\E\int_{\tau(i-1)}^{\tau i}\sum_k\|D^{\gamma}G^k_s\|^2\,ds\leq
\E\int_{\tau(i-1)}^{\tau i}\sum_k\|G^k_s\|^2_{\varrho}\,ds.
\end{equation}
By It\^o's identity for stochastic integrals 
$$
\E I_i^{(5)}=2\E\int_{\tau(i-1)}^{\tau i}
\int_{\bR^d}D^{\gamma}\frM_{i-1}^ke_{i-1}(x)D^{\gamma}G^k_s(x)\,dx\,ds. 
$$
Here by integration by parts we drop one derivative from 
$D^{\gamma}\frM_{i-1}^ke_{i-1}$ on the term $D^{\gamma}G^k_s$,  
and then by the Cauchy-Schwarz-Bunyakovsky and Young inequalities 
we get
$$
\E I_i^{(5)}\leq \tau N\E\|e_{i-1}\|^2_{\varrho}
+N\E\int_{\tau(i-1)}^{\tau i}\sum_k\|G^k_s\|^2_{\varrho+1}\,ds.
$$
Using \eqref {2.9.2}, by Theorem \ref{theorem Gy} (a) we have 
\begin{equation}                                                                            \label{est rev1}
\E I_i^{(1)}+\E I_i^{(4)}\leq N\tau\E\|e_i\|_{\varrho}^2.
\end{equation}
Therefore, by taking expectations and summing up 
\eqref{error estimate 1}  over $i$ from $1$ to $j\leq n$, and 
over $\gamma$ for $|\gamma|\leq l$, we get
\begin{equation}                                          \label{gronwall}
\E\|e_j\|_{l}^2\leq N_0\tau\sum_{i=1}^j\E\|e_i\|_{l}^2
+N_0\E\int_0^T(\|F_s\|^2_{l}+\sum_{k}\|G^k_s\|^2_{l+1})\,ds
\end{equation}
for $j=1,...,n$, where $N_0=N_0(K,C,\Lambda_0,l,d)$ is a constant.
Notice that due to \eqref{veges1} and \eqref{veges2}
$$
\E\|e_i\|^2_{l}<\infty\quad i=1,2,...,n, 
$$
and due to \eqref{veges2} 
and Assumptions \ref{as2:regularity} and \ref{assumption holder} 
we have 
$$
\E\int_{0}^T\|F_s\|_{l}^2+\sum_{k}\|G_s^k\|^2_{l+1}\,ds<\infty. 
$$
Hence the right-hand side of inequality \eqref{gronwall} is finite. 
Thus when $\tau<1/N_0$, from \eqref{gronwall}  by discrete Gronwall's lemma   
it follows that   
\begin{equation}                                                           \label{gronwall1}
\E\|e_j\|_{l}^2\leq N_0(1-N_0\tau)^{-j}
\E\int_0^T(\|F_t\|^2_{l}+\sum_{k}\|G^k_t\|^2_{l+1})\,dt
\end{equation}
for $j=1,...,n$.  
Now we are going to show that 
\begin{equation}                                                                  \label{1.7.2}
\E\|F_t\|^2_{l}+\sum_{k}\E\|G^k_t\|^2_{l+1}
\leq N\tau(\E\cK^2_{l+3}+1)  
\end{equation}
for all $t\in[0,T]$ with a constant $N=N(K,C,c,d,T,l,\Lambda_0)$. 
To estimate $\E\|F_t\|_{\varrho}^2$, first notice that due to 
Assumption \ref{assumption holder},
\begin{equation}                                              \label{holder f}
\E\|{f}^{R}_t-{f}^{R}_{\kappa_1(t)}\|_{l}^2\leq N\tau, 
\end{equation}
and  due to Assumptions 
\ref{assumption cutoff}, 
\ref{as2:regularity} 
 and 
\ref{assumption holder}
$$
\|L^{h,R}_tu^{h,R}_t-L^{h,R}_{\kappa_2(t)}u^{h,R}_{\kappa_2(t)}\|_{l}^2
\leq 2A_1(t)+2A_2(t)
$$
with 
$$
A_1(t):=\|(L^{h,R}_t-L^{h,R}_{\kappa_2(t)})u_{\kappa_2(t)}^{h,R}\|_{l}^2 
\leq N\tau\|{u}_{\kappa_2(t)}^{h,R}\|_{l+2}^2, 
$$
$$
A_2(t):=\|L_t^{h,R}(u_t^{h,R}-u_{\kappa_2(t)}^{h,R})\|_{l}^2
\leq 
N\|u^{h,R}_{\kappa_2(t)}-u^{h,R}_t\|_{l+2}^2.
$$
By virtue of Theorem (b) and Assumption \ref{assumption cutoff},
\begin{equation}                                            \label{2.7.2}
\E A_1(t)\leq N\tau \sup_{t\in[0,T]}
\E\|{u}_t^{h,R}\|_{l+2}^2\leq N\tau \E\cK^2_{l+2}. 
\end{equation}
To estimate $A_2$ we  show that 
\begin{equation}                                                \label{lipschitz}
\E\|u^{h,R}_t-u^{h,R}_s\|_{l+2}^2
\leq N|t-s|(1+\E\cK^2_{l+3})\quad\text{for all $s,t\in[0,T]$}.
\end{equation}
To this end we mollify $u^{h,R}$ in the spatial variable. We take a nonnegative 
$\phi\in C_0^{\infty}(\bR^d)$  supported on the unit ball 
in $\bR^d$ such that it has unit integral, 
and for functions $\varphi$, which are locally integrable in $x\in\bR^d$, 
we define $\varphi^{(\varepsilon)}$ by 
$$
\varphi^{(\varepsilon)}(x):
=\varepsilon^{-d}\int_{\bR}\varphi(y)\phi((x-y)/\varepsilon)\,dy, 
\quad x\in\bR^d, 
\quad\text{for $\varepsilon>0$.} 
$$
We will make use of the following known and easily verifiable properties 
of $\varphi^{(\varepsilon)}$. 
For integers $r\geq0$ and $\varepsilon>0$, 
\begin{equation}                                           \label{hasznos1}
 \|\varphi^{(\varepsilon)}-\varphi\|_{r}\leq \varepsilon\|\varphi\|_{r+1} 
\quad \text{for $\varphi\in W^{r+1}_2$}, 
\end{equation}
and 
\begin{equation}                                                 \label{hasznos2}
\|\varphi^{(\varepsilon)}\|_r\leq 
\|\varphi\|_r,\quad\|(D_i\varphi)^{(\varepsilon)}\|_{r}
=\|D_i\varphi^{(\varepsilon)}\|_{r}
\leq \frac{N_i}{\varepsilon}\|\varphi\|_{r}, \quad \text{$\varphi\in W^{r}_2$} 
\end{equation}
for $i=1,2,...,d$, where $N_i$ depends only 
on the sup and $L_1$ norms of $D_i\phi$. 
Thus by \eqref{hasznos1}
$$
\|u^{h,R}_{t}-u^{h,R}_s\|_{l+2}^2
=
\|u^{h,R}_{t}-u^{h,R}_s\|_{l+1}^2
+\sum_{i=1}^d\|D_i(u^{h,R}_{t}-u^{h,R}_s)\|_{l+1}^2
$$
$$
\leq \|u^{h,R}_{t}-u^{h,R}_s\|_{l+1}^2
+\sum_{i=1}^d\|D_i(u^{h,R}_{t}-u^{h,R}_s)^{(\varepsilon)}\|_{l+1}^2
$$
$$
+\sum_{i=1}^d\|D_iu^{h,R}_{t}-(D_iu^{h,R}_t)^{(\varepsilon)}\|^2_{l+1}
+\sum_{i=1}^d\|(D_iu^{h,R}_s-(D_iu^{h,R}_s)^{(\varepsilon)}\|^2_{l+1}
$$
\begin{align}                                                            \label{1.8.2}                                               
\leq &\|u^{h,R}_{t}-u^{h,R}_s\|_{l+1}^2      
+\sum_{i=1}^d\|D_i(u^{h,R}_{t}-u^{h,R}_s)^{(\varepsilon)}\|_{l+1}^2 \nonumber\\
&+N\varepsilon^{2}(\|u^{h,R}_{t}\|_{l+2}
+\|u^{h,R}_s\|_{l+2}). 
\end{align}
Since $u^{h,R}$ satisfies \eqref{eq:finite diff}-\eqref{eq:finite diff-ini}, 
for $0\leq s\leq t\leq T$ we have
\begin{align}                                                                      \label{2.8.2}
\|u^{h,R}_{t}-u^{h,R}_s\|_{l+1}^2\leq 2B_1+2B_2
\end{align}
with 
$$
B_1:=\Big\|\int_s^{t}(L^{h,R}_ru^{h,R}_r+f_r^R)\,dr\Big\|^2_{l+1}, 
$$
$$
B_2:=
\Big\|\int_s^{t}(M^{h,R,k}_ru^{h,R}_r+g^{R,k}_r)\,dw^k_r\Big\|_{l+1}^2,  
$$
and using \eqref{hasznos2} we have 
\begin{equation}                                                                          \label{3.8.2}
\sum_{i=1}^d\|D_i(u^{h,R}_{t}-u^{h,R}_s)^{(\varepsilon)}\|_{l+1}^2 
\leq \frac{N}{\varepsilon^2}B_1+B_3, 
\end{equation}
with 
$$
B_3:=
\sum_{i=1}^d
\Big\|\int_s^{t}(D_iM^{h,R,k}_ru^{h,R}_r
+D_ig^{R,k}_r)\,dw^k_r\Big\|_{l+1}^2. 
$$
It is easy to see that 
$$
\E B_1\leq (t-s)\int_s^{t}\E\|L^{h,R}_su^{h,R}_s
+f_s^R\|^2_{l+1}\,ds
$$
\begin{equation}                                                          \label{4.8.2}
\leq N(t-s)^2(\sup_{t\leq T}\E\|u_t\|^2_{l+3}
+\sup_{t\leq T}\E\|f_t\|^2_{l+1}). 
\end{equation}
Using It\^o's identity, for $B_2$  we obtain
\begin{align}
\E B_2=&\E\int_s^{t}
\sum_{k}\|M^{h,R,k}_ru^{h,R}_r+g^k_r\|^2_{l+1}\,dr             \nonumber\\
&\leq N(t-s)(\sup_{t\leq T}\E \|u_t^{h,R}\|^2_{l+2}
+\sup_{t\leq T}\E\|g_t\|^2_{l+1}),                                            \label{5.8.2}
\end{align}
and in the same way, for $B_3$ we have 
\begin{equation}                                                                           \label{6.8.2}
\E B_3 \leq N(t-s)(\sup_{t\leq T}\E \|u_t^{h,R}\|^2_{l+3}
+\sup_{t\leq T}\E\|g_t\|^2_{l+2}). 
\end{equation}
From  \eqref{1.8.2} through \eqref{5.8.2} we get 
\begin{equation*}                                                     
\E\|u_t^{h,R}-u_s^{h,R}\|^2_{l+2}\leq N(\varepsilon^2+(t-s)^2\varepsilon^{-2}+|t-s|)J
\end{equation*}
for every $\varepsilon>0$, where  
\begin{equation*}   
J:=\sup_{t\leq T}\{\E\|u_t^{h,R}\|^2_{l+3}+
\E\|g_t\|^2_{l+2}+\E \|f_t\|^2_{l+1}\}. 
\end{equation*}
Choosing here $\varepsilon:=|t-s|^{1/2}$ gives 
\begin{equation*}                                                        \label{7.8.2}
\E\|u_t^{h,R}-u_s^{h,R}\|^2_{l+2}
\leq 3N(t-s)J. 
\end{equation*}
Hence we obtain \eqref{lipschitz},  
since by 
Theorem \ref{theorem Gy} and Assumption \ref{assumption holder} we have 
$$
J\leq N\E\cK^2_{l+3}+H. 
$$
From \eqref{lipschitz} we get the estimate 
$$
\E A_1\leq N\tau(\E\cK^2_{l+3}+1), 
$$
which together with \eqref{2.7.2} and \eqref{holder f} 
shows that 
$\E\|F_t\|^2_{l}$ is estimated by 
the right-hand  side of the inequality \eqref{1.7.2}. 
Similarly,  by making use of \eqref{lipschitz}, we can get  
the same estimate for $\E\|G_t\|^2_{l+1}$, 
 which finishes the proof of \eqref{1.7.2}. 
From \eqref{gronwall1} and \eqref{1.7.2}  we have 
\begin{equation*}
\max_{j\leq n}\E\|e_j\|_{\varrho}^2\leq \tau N(1-N_0\tau)^{-n}(\E\cK_{\varrho+3}+1)
\end{equation*}
for $\tau<N_0^{-1}$, with a constant $N=N(K,C,H,d,l,T,\Lambda_0)$. 
Hence noticing that  
$$
\lim_{n\to\infty}(1-N_0\tau)^{-n}=e^{N_0T}, 
$$
we obtain \eqref{gyenge becsles}, 
which completes the proof of (ii). 

To prove (iii) notice first that by using (ii) with $l+1$ in place of $l$ 
we have 
\begin{equation}                                                            \label{1.10.2}
\E\sum_{i=1}^{n}\|e_i\|^2_{l+1}\leq N(1+\E\cK_{m}^2)
\end{equation}
with $m=l+4$.
Fixing a multi-index $\gamma$ as in (ii), with 
$|\gamma|=\varrho\leq l$, we revisit \eqref{error estimate 1} 
and observe that $I_i^{(5)}\leq I_i^{(4)}+I_i^{(6)}$  
and by Theorem \ref{theorem Gy} 
$$
I^{(1)}_i\leq N\tau\|e_i\|^2_{l}. 
$$
Thus summing up \eqref{error estimate 1} over $i=1,...,j$, 
and recalling also \eqref{est rev0} and \eqref{2.9.2}, we obtain 
$$
\E\max_{1\leq j\leq n}\|D^{\gamma}e_i\|^2\leq N\tau\E\sum_{i=1}^n\|e_i\|^2_l
+N\int_0^T\E\|F_s\|^2_{l}\,ds
$$
\begin{equation*}                                                          \label{4.10.2}
+N\int_0^T\E \|G_s\|^2_{l}\,ds+
\tau\sum_{i=1}^n\E\sum_{k}\|D^{\gamma}\frM_{i-1}^ke_{i-1}\|^2+\E\max_{1\leq j\leq n}\sum_{i=1}^jI^{(3)}_i.  
\end{equation*}
Hence, noticing that 
$$
\sum_{k=1}^{\infty}\|D^{\gamma}\frM_{i-1}^ke_{i-1}\|^2
\leq N\|e_i\|^2_{l+1},  
$$
by using \eqref{1.10.2} and  \eqref{1.7.2},  
we get 
\begin{equation}                                             \label{5.10.2}
\E\max_{1\leq i\leq n}\|D^{\gamma}e_i\|^2\leq N\tau(\E\cK^2_m+1)
+\E\max_{1\leq j\leq n}\sum_{i=1}^jI^{(3)}_i.
\end{equation}
Clearly,
$$
\sum_{i=1}^jI^{(3)}_i= \int_0^{j\tau}Z^k_t\,dw^k_t, \quad j=1,2,...,n, 
$$
where $(Z^k_t)_{t\in[0,T]}$ is defined by 
$$
Z^k_t=2(D^{\gamma}e_{i-1},D^{\gamma}\frM_{i-1}^ke_{i-1}+D^{\gamma}G^k_t)
\quad\text{for $t\in(t_{i-1}, t_i]$, 
$i=1,2,...,n$}. 
$$
It is easy to see that 
$$
\int_0^T\sum_{k}|Z^k_t|^2\,dt
\leq4\max_{1\leq i\leq n}\|D^{\gamma}e_i\|^2
\sum_{k}
\left(
\sum_{i=1}^n
\|D^{\gamma}\frM_{i-1}^ke_{i-1}\|^2\tau+\int_0^T\|G^k_s\|^2_{l}\,ds
\right)
$$
$$
\leq\frac{1}{36}\max_{1\leq i\leq n}\|D^{\gamma}e_i\|^4
+144
\left(
\sum_{i=1}^n
\|e_{i-1}\|_{l+1}^2\tau
+\sum_k\int_0^T\|G^k_s\|^2_{l}\,ds
\right)^2.
$$
Hence by the Davis inequality we have
$$
S_n\leq \E\sup_{t\in[0,T]}\int_0^tZ^k_s\,dw_s^k
\leq 3\E\left(\int_0^T\sum_{k}|Z^k_s|^2\,ds\right)^{1/2}
$$
$$
\leq \frac{1}{2}\E\max_{1\leq i\leq n}\|D^{\gamma}e_i\|^2+
36\tau\sum_{i=1}^n\E\|e_{i-1}\|_{l+1}^2
+36\E\int_0^T\|G_s\|^2_{l}\,ds<\infty. 
$$
Using this, and estimates \eqref{1.7.2} and \eqref{1.10.2}, 
we obtain (iii) from \eqref{5.10.2}.
\end{proof}
\medskip

By virtue of Remark \ref{remark important}, 
by using Sobolev's theorem on embedding $W^m_2$ into $C_b$ for $m>d/2$,  
we get the following corollary. 
\begin{corollary}                                                       \label{corollary time error}
(i) Let Assumptions \ref{assumption cutoff} through 
 \ref{assumption free} hold with $m>d/2$.  
Then for sufficiently small $\tau$ there exists for all $h$ and $R>0$ 
a unique solution 
$v^{h,R,\tau}=\{v^{h,R,\tau}_{i\tau}(x):x\in \G_h\cap\rm{supp}\,\zeta_R\}$ to 
\eqref{eq:discrete time-ini}-\eqref{eq:discrete time}. 

\noindent
(ii) If Assumption \ref{assumption holder} 
holds with some integer $l> d/2$ and Assumptions \ref{as2:compatibility}   
through 
 \ref{assumption free}
hold with $m>d/2+3$, then 
\begin{equation}                                 \label{gyenge becsles2}
\max_{0\leq i\leq n}\E\max_{x\in \bG_h}
|u^{h,R}_{\tau i}-v^{h,R,\tau}_i|^2
\leq N\tau(1+\E\cK_{m}^{2}),
\end{equation}
with a constant 
$N=N(K,C,H,d,T,l,\Lambda_0)$.  

\noindent
(iii) If  Assumption \ref{assumption holder} 
holds with $l>d/2+1$ and Assumptions \ref{as2:compatibility}   
through 
 \ref{assumption free}
hold with $m>d/2+4$, then 
\begin{equation} 						\label{erosebb becsles2}
\E \max_{i\leq n}\max_{x\in \bG_h}
|u^{h,R}_{\tau i}-v^{h,R,\tau}_i|^2
\leq N\tau(1+\E\cK_{m}^{2})
\end{equation}
with a constant 
$N=N(K,C,H,d,T,l,\Lambda_0)$.  
\end{corollary}

\medskip
{\it Proof of Theorems \ref{theorem main1} and \ref{theorem main2}}. 
The solvability of \eqref{eq:discrete time-ini}-\eqref{eq:discrete time} has already been discussed above, so we need only prove the estimates \eqref{eq:full discret error} and \eqref{eq:full discret accelerated}. A natural way to separate the errors of the different types of approximations would be to write
$$
|u_{\tau i}(x)-v^{h,R,\tau}_i(x)|\leq|u_{\tau i}(x)-u_{\tau i}^h(x)|
+|u_{\tau i}^{h}(x)-u_{\tau i}^{h,R}(x)|+|u_{\tau i}^{h,R}(x)-v^{h,R,\tau}_i(x)|.
$$
However, in such a decomposition we cannot directly estimate the middle term on the right-hand side, that is, the localization error of the finite difference equation. Therefore we introduce $u^{0,R}$, the classical solution of \eqref{SPDE general}-\eqref{SPDE general-ini} with data
$$
\bar \frD:=\frD^R=(\zeta_R\psi,\zeta_R^2a,\zeta_Rb,\zeta_Rc,\zeta_R\sigma,\zeta_R\mu,\zeta_Rf,\zeta_Rg).
$$
Clearly the pair $\frD$, $\bar\frD$ satisfies Assumption \ref{assumption R}. Also, as the finite difference coefficients $\fra,\frb$ are compatible with the data $\frD$ in the sense that Assumption \ref{as2:compatibility} is satisfied, it follows that the coefficients $\fra^R,\frb^R$, as defined in \eqref{coeffcutoff}-\eqref{coeffcutoff2}, are compatible with the data $\bar\frD$ in the same sense. Therefore in the decomposition
$$
|u_{\tau i}(x)-v^{h,R,\tau}_i(x)|\leq|u_{\tau i}(x)-u_{\tau i}^{0,R}(x)|
+|u_{\tau i}^{0,R}(x)-u_{\tau i}^{h,R}(x)|+|u_{\tau i}^{h,R}(x)-v^{h,R,\tau}_i(x)|
$$
the first term can be treated by Theorem \ref{theorem cutoff}, 
the second by Theorem \ref{theorem Gy} (c), and the third by 
Corollary \ref{corollary time error} (iii). 
Adding up the resulting errors we get the estimate \eqref{eq:full discret error}, and the dependence of the constants also follows from the invoked theorems.

Similarly, for \eqref{eq:full discret accelerated} we write
\begin{align*}
|u_{\tau i}(x)-\bar v^{h,R,\tau}_i(x)|&\leq|u_{\tau i}(x)-u_{\tau i}^{0,R}(x)|
+|u_{\tau i}^{0,R}(x)-\bar u_{\tau i}^{h,R}(x)|\\&+\sum_{j=0}^r c_j|u_{\tau i}^{h/2^j,R}(x)-v^{h/2^j,R,\tau}_i(x)|,
\end{align*}
where $\bar u^{h,R}=\sum_{j=0}^r c_j u^{h/2^j,R},$ and follow the same steps as above.

\medskip
\begin{remark}
As it can be easily seen from the last step of the proof, 
Assumption \ref{assumption holder} can be weakened to $\alpha$-H\"older continuity for any fixed $\alpha>0$, at the cost of lowering the rate from $1/2$ to $\alpha\wedge(1/2)$.

To decrease the spatial regularity conditions, in particular, the term $d/2$, one would need the generalization of the results of \cite{Gy}, and subsequently, of Theorem \ref{theorem time error}, to arbitrary Sobolev spaces $W^m_p$. Partial results in this direction can be found in \cite{GG}.
\end{remark}

\bigskip
\textbf{Acknowledgements.} The authors thank the referees for their remarks and suggestions which helped to improve the presentation of the paper.

\bibliographystyle{amsplain}

\end{document}